\documentclass[10pt,a4paper]{article}

\usepackage{microtype}
\usepackage[colorlinks={true},linkcolor={blue},citecolor={blue}, urlcolor={blue} ]{hyperref}
\usepackage{tocbibind}
\usepackage{stmaryrd}   
\usepackage{amsmath}
\usepackage{verbatim}
\usepackage{mathrsfs}
\usepackage{mathtools}
\usepackage{amsthm} 
\usepackage{amssymb} 
\usepackage{mdframed}
\usepackage[margin=2.5 cm]{geometry}
\usepackage{enumerate}
\usepackage{bm}

\usepackage{tikz}
\usepackage{tikz-cd}
\usetikzlibrary{intersections}
\usetikzlibrary{spy} 
\usetikzlibrary{decorations.pathmorphing}
\usetikzlibrary{decorations.pathreplacing}

\newtheoremstyle{myremark}
  {\topsep}   
  {\topsep}   
  {}  
  {0pt}       
  {\itshape} 
  {.}         
  {5pt plus 1pt minus 1pt} 
  {} 

\newtheorem{thm}{Theorem}[section]
\newtheorem{lem}[thm]{Lemma}
\newtheorem{prop}[thm]{Proposition}
\newtheorem*{thm*}{Theorem}
\newtheorem{cor}[thm]{Corollary}

\theoremstyle{definition}
\newtheorem{defin}[thm]{Definition}
\theoremstyle{myremark}
\newtheorem{rem}[thm]{Remark}

\newtheorem{exa}[thm]{Example}
\newtheorem*{nota}{Notation}

\newcommand{\hooklongrightarrow}{\lhook\joinrel\longrightarrow}

\DeclareMathOperator{\J}{J}
\DeclareMathOperator{\alb}{Alb}

\DeclareMathOperator{\im}{Im}

\DeclareMathOperator{\spec}{Spec}

\DeclareMathOperator{\tr}{tr}
\DeclareMathOperator{\Pic}{Pic}

\DeclareMathOperator{\Div}{Div}

\DeclareMathOperator{\rk}{rk}

\DeclareMathOperator{\divi}{div}

\DeclareMathOperator{\CH}{CH}
\DeclareMathOperator{\ver}{Vert}
\DeclareMathOperator{\NS}{NS}

\DeclareMathOperator{\Ima}{Im}
\DeclareMathOperator{\Num}{Num}
\DeclareMathOperator{\Rat}{Rat}
\DeclareMathOperator{\alg}{alg}
\DeclareMathOperator{\an}{an}
\DeclareMathOperator{\bPic}{\mathbf{Pic}}
\DeclareMathOperator{\Br}{Br}

\newcommand{\sbt}{\,\begin{picture}(-1,1)(-1,-3)\circle*{3}\end{picture}\ }

\makeatletter

\title{Numerical equivalence of $\mathbb R$-divisors and Shioda-Tate formula for arithmetic varieties}

\author{
  Paolo Dolce
  \and
  Roberto Gualdi
}
\date{}

\newcommand{\Addresses}{{
  \bigskip
  \footnotesize

  }
  
  P.~Dolce, \textsc{University of Udine}\par\nopagebreak
  \textit{E-mail address}: \texttt{paolo.dolce@uniud.it}
  
  \medskip
  
  R.~Gualdi, \textsc{Fakult\"at f\"ur Mathematik, Universit\"at Regensburg}\par\nopagebreak
  \textit{E-mail address}: \texttt{roberto.gualdi@mathematik.uni-regensburg.de}

}

\begin{document}

\maketitle

\makeatletter
\@starttoc{toc}
\makeatother
\abstract{Let $X$ be an arithmetic variety over the ring of integers of a number field~$K$, with smooth generic fiber~$X_K$.
We give a formula that relates the dimension of the first Arakelov-Chow vector space of $X$ with the Mordell-Weil rank of the Albanese variety of $X_K$ and the rank of the N\'eron-Severi group of~$X_K$.
This is a higher dimensional and arithmetic version of the classical Shioda-Tate formula for elliptic surfaces.

Such analogy is strengthened by the fact that we show that the numerically trivial arithmetic $\mathbb{R}$-divisors on $X$ are exactly the linear combinations of principal ones.
This result is equivalent to the non-degeneracy of the arithmetic intersection pairing in the argument of divisors, partially confirming \cite[Conjecture~1]{GS94}.}

\setcounter{section}{-1}
\section{Introduction}
Given a smooth elliptic surface $Y\to C$ over an algebraically closed field~$k$, admitting a section and with generic fiber $E$ (which is an elliptic curve over~$k(C)$), the \emph{Shioda-Tate formula} compares the rank of the N\'eron-Severi group of $Y$ with the rank of the group of the $k(C)$-rational points of~$E$.
The explicit relation (see \cite[Corollary~1.5]{Shi72} or also \cite[page~429]{Tat95} for the analogous formula over a finite base field) reads
\begin{equation}\label{sh-t}
\rk(\NS(Y))=\rk(E(k(C))+2+\sum_{c\in C} (f_c-1)\,,
\end{equation}
where $f_c$ is the number of irreducible components of the fiber~$Y_c$.
Similar geometric formulas have been recently found for higher dimensional fibrations: see for example \cite{HPW05}, \cite{Ogu09},~\cite{Kah09}.

\begin{rem}\label{rem1}
Note that the rank of the N\'eron-Severi group $\NS(Y)$ coincides with the one of the numerical Picard group~$\Num(Y)$, i.e. the group of divisors of $Y$ modulo numerical equivalence, see for instance \cite[Example~19.3.1]{Ful98}.
As a consequence, one can interpret \eqref{sh-t} as a relation between the set of $k(C)$-rational points of $E$ and the intersection theoretic features of its model $Y$ over~$C$. 
\end{rem}

Thanks to the analogy between function fields and number fields, one expects an arithmetic version of \eqref{sh-t} to hold for models of curves over a number field~$K$, and moreover it is natural to investigate the arithmetic higher dimensional setting.
The study of intersection theory on completed arithmetic surfaces was initiated by S.~J. Arakelov in~\cite{Ara74}, which introduced the notion of divisors in this framework.
Such theory was later deepened by G. Faltings in \cite{Fal83} and \cite{Fal84} with a proof of the Mordell conjecture and of a version of Riemann-Roch theorem for arithmetic surfaces.
In the nineties, H. Gillet and C. Soul\'e in \cite{GS90a} and \cite{GS90b} further extended the theory to arithmetic varieties of any dimension by introducing the notion of arithmetic Chow ring, as we recall in subsection~\ref{subsection_1.1}.

\vspace{\baselineskip}

Let hence $K$ be a number field and $X\to B=\spec O_K$ be a projective arithmetic variety of relative dimension $d$ such that the generic fiber $X_K$ is a smooth, geometrically integral variety over~$K$.
Then, after tensorizing with $\mathbb R$ only the finite part of the arithmetic first Chow group, one can define the notion of arithmetic $\mathbb R$-divisor of $X$ (see~\cite{GS94}) and declare that an arithmetic $\mathbb R$-divisor is \emph{numerically trivial} if its intersection pairing with any arithmetic $\mathbb R$-cycle of complementary dimension is~$0$.
In Theorem~$\ref{hi_bo}$ we show that

\begin{thm*}\label{intro_hi_bo}
An arithmetic $\mathbb R$-divisor is numerically trivial if and only if it is a linear combination of principal ones.
\end{thm*}

Therefore we bring to light an important difference with algebraic geometry where, on the  contrary, it is customary to find numerically trivial divisors which are not principal.
In our proof we make use of the arithmetic higher dimensional Hodge index theorems for arithmetic divisors and~$\mathbb R$-divisors, which are proved by A. Moriwaki respectively in \cite{Mor96} and \cite{Mor13} (see also \cite{YZ17}) and of the Dirichlet's unit theorem.

The arithmetic intersection pairing for arithmetic cycles of codimension~$q$, in symbols
\[
\left<\;,\,\right>\colon\widehat{\CH}_{\mathbb{R}}^q(X)\times\widehat{\CH}_{\mathbb{R}}^{d+1-q}(X)\longrightarrow\mathbb{R}\,,
\]
was conjectured to be non-degenerate by H. Gillet and C. Soul\'e in \cite[Conjecture~1]{GS94}.
Notice that our Theorem~\ref{hi_bo} confirms this conjecture in the first argument when~$q=1$.

The result of Theorem~\ref{hi_bo} moreover asserts that the knowledge of the intersection-theoretic behaviour of an arithmetic $\mathbb{R}$-divisor is enough to reconstruct its class, and suggests that an analogous of \eqref{sh-t} for arithmetic varieties should involve the dimension of the first Arakelov-Chow space.
Indeed, G. Faltings showed in \cite[Theorem~4.(d)]{Fal84} that when $X$ is an arithmetic surface:
\begin{equation}\label{fal-sh-t}
\dim(\overline{\CH}^1_{\mathbb R}(X))=\rk(\J(X_K)(K))+2+\sum_{\mathfrak p\in B} (f_{\mathfrak p}-1)\,,
\end{equation}
where $\overline{\CH}^1_{\mathbb R}(X)$ is the vector space of Arakelov $\mathbb R$-divisors modulo rational equivalence, $f_{\mathfrak p}$ is the number of irreducible components of the fiber $X_{\mathfrak p}$, and $\J(X_K)$ is the Jacobian variety of the curve~$X_K$.
The key result used for the proof of \eqref{fal-sh-t} is the arithmetic Hodge index theorem for arithmetic surfaces, see \cite{Fal84},~\cite{Hri85}.
Keeping Remark~\ref{rem1} and Theorem~\ref{hi_bo} in mind, one can immediately notice the striking similarities between \eqref{sh-t} and~\eqref{fal-sh-t}.
Driven by this comparison, we prove in Theorem~\ref{hi_ST} the following higher dimensional arithmetic version of the Shioda-Tate formula, without making any assumption on the existence of a $K$-rational point of the generic fiber.

\begin{thm*}
For an arithmetic variety $X$ (as above) of arbitrary dimension:
\begin{equation}\label{our-sh-t}
\dim(\overline{\CH}^1_{\mathbb R}(X))=\rk(\alb(X_K)(K))+\rk(\NS(X_K))+1+\sum_{\mathfrak p\in B} (f_{\mathfrak p}-1)\,.
\end{equation}
\end{thm*}

The ingredients appearing in the previous formula deserve a quick explanation:
\begin{itemize}
    \item[\sbt] $\overline{\CH}^1_{\mathbb R}(X)$ is the vector space of Arakelov $\mathbb R$-divisors on~$X$ modulo rational equivalence. We recall that an Arakelov $\mathbb R$-divisor is an arithmetic $\mathbb{R}$-divisors with harmonic curvature, see \cite[section~5.1]{GS90a}.
    \item[\sbt] $\alb(X_K)$ is the Albanese variety of~$X_K$, and it is well known by the Mordell-Weil theorem (see for example~\cite{LN59}) that $\alb(X_K)(K)$ is a group of finite rank.
    \item[\sbt] $\NS(X_K)$ is the N\'eron-Severi group of $X_K$.
    \item[\sbt] $f_{\mathfrak p}$ is the number of irreducible components of the fiber~$X_{\mathfrak p}$, as above.
    Remember that $f_{\mathfrak p}=1$ for all but finitely many primes~$\mathfrak p$. 
\end{itemize}

Notice that, when $X_K$ is a curve, \eqref{our-sh-t} coincides with~\eqref{fal-sh-t}, since the group~$\NS(X_K)$ has rank~$1$.

\paragraph{Acknowledgements.}  The research activities of the first author, in relation to this work, were mostly supported by the \emph{Italian national grant ``Ing. Giorgio Schirillo"} conferred by \emph{INdAM}.  The first author wants to thank \emph{M. Tamiozzo} for several interesting discussions on the subject originated from his talk of June 2017 in Nottingham, regarding the rank of the Picard group of arithmetic elliptic surfaces. He also espresses his gratitude to \emph{P. Corvaja}, \emph{I. Fesenko} and \emph{F. Zucconi}.

The second author has benefited from the generous funding of the \emph{Alexander von Humboldt Foundation} while preparing this work, as well as from the \emph{SFB 1085: Higher Invariants} and from the MINECO research project \emph{MTM2015-65361-P}, and he thanks them heartily for their support.
He also wishes to thank \emph{C. Mart{\'i}nez} and \emph{M. Stoffel} for some useful conversations.

The authors are also indebted to \emph{W. Gubler} for comments on a first draft of the paper and to \emph{K. K\"unnemann} for pointing out the connection to the arithmetic standard conjectures.
Finally, they wish to express their thankful gratitude to the anonymous referee, whose careful reading has allowed a notable improvement of the text.

\section{Arithmetic intersection theory}

Let $K$ be a number field with ring of integers~$O_K$, and let $\Sigma$ denote the set of embeddings~$K\hookrightarrow\mathbb{C}$.
We fix for the entire paper a \emph{projective arithmetic variety} $X$ over~$O_K$, that is a regular, flat, integral projective scheme of finite type over~$B:=\spec O_K$ with smooth and geometrically integral generic fiber~$X_K$.
It has geometrically connected fibers because of~\cite[Lemma~37.48.6]{Sta-Pr}.
We denote by $\pi\colon X\to B$ its structure map, and by $d$ its relative dimension.
For any embedding~$\sigma$, we write $X_\sigma$ for the $\mathbb{C}$-scheme~$X\times_\sigma\spec\mathbb{C}$.
The complex analytic spaces associated to $X_\sigma$ and $X_{\overline{\sigma}}$ are related by a natural morphism $F_{\sigma}$ induced by complex conjugation.
As general notation we always attach the subscript $K$ to a cycle on $X$ in order to denote its restriction to the generic fiber~$X_K$.

This section is aimed at reviewing the definition and study of the arithmetic intersection theory on~$X$, following the classical reference~\cite{GS90a}, but working instead with $\mathbb{R}$-cycles as done in~\cite{GS94}.
This will offer us the occasion to fix the notations used throughout all the paper and recall the results which will be brought into play during the proof of our main contributions.

\subsection{The intersection pairing for arithmetic \texorpdfstring{$\mathbb{R}$}{R}-cycles}\label{subsection_1.1}

Let~$q\in\{0,\ldots,d+1\}$.
If $Y=\sum a_iY_i$ is a formal real linear combination of integral subschemes of~$X$ of codimension~$q$, we write $Y_\sigma$ for the corresponding $\mathbb{R}$-cycle of~$X_\sigma$, and $\delta_{Y_\sigma}$ for the integration current on the associated analytic $\mathbb{R}$-cycle.
A \emph{Green current} for $Y_\sigma$ is a  current $g_\sigma$ on $X_\sigma^{\an}$ of bidegree~$(q-1,q-1)$, of real type (i.e. satisfying~$F^{\ast}_\sigma(g_{\sigma})= (-1)^{q-1}g_{\bar\sigma}$) and such that \[
dd^cg_\sigma+\delta_{Y_\sigma}=[\omega(Y,g_\sigma)]
\]
for  a $\mathscr{C}^\infty$-form  $\omega(Y,g_\sigma)$ of bidegree~$(q,q)$  (that is uniquely determined).
Here and for the whole paper, $[\omega]$ stands for the current associated to a smooth form~$\omega$.

\begin{defin}\label{defn arithmetic cycles}
An \emph{arithmetic $\mathbb{R}$-cycle of codimension $q$ on $X$} is a pair~$(Y,(g_\sigma)_{\sigma\in\Sigma})$, where $Y$ is a real linear combination of $q$-codimensional integral subschemes of~$X$, and $g_\sigma$ is a Green current for $Y_\sigma$ for all~$\sigma\in\Sigma$.
The set of arithmetic $\mathbb{R}$-cycles of codimension $q$ on $X$ is equipped with a natural structure of $\mathbb{R}$-vector space, and it is denoted by~$\widehat{Z}^q_\mathbb{R}(X)$.
An element of $\widehat{Z}^1_\mathbb{R}(X)$ is simply called an \emph{arithmetic $\mathbb{R}$-divisor}.
\end{defin}

An arithmetic $\mathbb{R}$-cycle $(Y,(g_\sigma)_\sigma)$ is said to be \emph{vertical} if its geometric part $Y$ is so, that is if the support of $Y$ is not surjectively mapped to $B$ by~$\pi$ or, equivalently, if the cycle $Y_K$ it induces on the generic fiber is the zero cycle.

For all $\mathfrak{p}\in B$ with~$\mathfrak{p}\neq(0)$, we denote by $X_{\mathfrak{p}}$ the cycle on $X$ associated to the fiber of $\pi$ over~$\mathfrak{p}$.
By~\cite[Lemma~1.7.1]{Ful98}, it coincides with the pull-back cycle~$\pi^\ast\mathfrak{p}$.
The subspace of $\mathbb R$-cycles of codimension $q$ on $X$ whose support is contained in $X_\mathfrak{p}$ is denoted by~$Z^q_\mathbb{R}(X)_{\mathfrak p}$.
In particular, $Z^1_\mathbb{R}(X)_{\mathfrak p}$ is the finite dimensional subspace generated by the irreducible components of the support of~$X_{\mathfrak p}$.

\begin{rem}\label{canon_image_vert_arithm}
There is a canonical linear map from the space of vertical $\mathbb{R}$-cycles to the space of arithmetic $\mathbb{R}$-cycles of~$X$, given by~$Y\mapsto(Y,(0)_\sigma)$.
In the following, we will often identify a vertical $\mathbb{R}$-cycle $Y$ of $X$ with its image under this map.
\end{rem}

\begin{rem}
Another natural equivalent possibility for the definition of $\mathbb{R}$-arithmetic cycles of $X$ is to consider Green currents, instead for the set of all embeddings of $K$ in~$\mathbb{C}$, just for their family up to conjugation (that is the set of infinite places of~$K$).
The formulas for the arithmetic intersection product and pairing that we recall in this subsection can be written in this other setting by distinguishing between real and complex embeddings of~$K$, see \cite[Appendix~B]{CD20} for the case of relative dimension~$d=1$.
To avoid confusions with our main reference, we nevertheless stick to the point of view of considering the whole family of complex embeddings of~$K$.

Notice however that, differently from~\cite{GS90a}, we do not define Green currents as classes modulo the sum by the images of $\partial$ and~$\bar{\partial}$.
This divergence turns harmless in our treatment, which is only concerned with arithmetic Chow spaces and arithmetic~$\mathbb{R}$-divisors.
\end{rem}

Analogously to the classical geometry setting, one can introduce a notion of rational equivalence for arithmetic $\mathbb{R}$-cycles on~$X$.
To do this, let $W$ be an integral subscheme of $X$ of codimension~$q-1$, and $f\in K(W)^\times$ a nonzero rational function on it.
By the use of Lelong-Poincar\'{e} formula, it is possible to define a Green current for $\divi(f)_\sigma$ for all~$\sigma\in\Sigma$, see \cite[section~3.3.3]{GS90a} for the precise definition, and hence obtain an arithmetic $\mathbb{R}$-cycle of codimension $q$ on~$X$, which is denoted by~$\widehat{\divi}(f)$.
When~$q=1$, this is achieved by setting, for all nonzero rational function~$f\in K(X)^\times$,
\[
\widehat{\divi}(f):=\left(\divi(f),(-\log|f_\sigma|^2)_\sigma\right)\,,
\]
where $f_\sigma$ is the analytic function induced by $f$ on the analytic space~$X_\sigma^{\an}$.

\begin{rem}\label{rem_div_const_expl}
Since $X$ is a scheme over~$\spec O_K$, an element $\xi\in O_K\setminus\{0\}$ defines a nonzero rational function on~$X$.
Its associated arithmetic $\mathbb{R}$-divisor can be written explicitly as
\[
\widehat{\divi}(\xi)= \left(\sum_{i=1}^me_iX_{\mathfrak{p}_i},\left(-\log|\sigma(\xi)|^2\right)_\sigma\right)\,,
\]
where $(\xi)=\mathfrak{p}_1^{e_1}\cdots\mathfrak{p}_m^{e_m}$ is the factorization of the ideal $(\xi)$ in~$O_K$.
The geometric part in the previous expression can be checked using the equality of cycles in the proof of~\cite[Proposition~2.3.(d)]{Ful98} and the definition of the divisor of $\xi$ on~$\spec O_K$.
In particular, $\widehat{\divi}(\xi)$ is vertical.
Notice that the geometric part of $\widehat{\divi}(\xi)$ vanishes if and only if~$\xi\in O_K^\times$.
More strongly, by Kronecker's theorem one has that $\widehat{\divi}(\xi)$ is the zero arithmetic cycle on $X$ if and only if $\xi$ is a root of unity in~$O_K$.
\end{rem}

With the previous notions in mind, we can introduce the arithmetic version (over~$\mathbb{R}$) of the classical geometrical Chow groups.

\begin{defin}\label{def_arith_CH}
We define $\widehat{\Rat}_{\mathbb R}^q(X)$ to be the real vector subspace of $\widehat{Z}^q_\mathbb{R}(X)$ generated by the arithmetic $\mathbb{R}$-cycles of codimension $q$ of $X$ of the form $\widehat{\divi}(f)$ and by those of the form~$(0,(\partial u_\sigma+\bar\partial v_\sigma)_\sigma)$, where $u_\sigma$ and $v_\sigma$ are currents of suitable bidegree.
The quotient
\[
\widehat{\CH}^q_{\mathbb{R}}(X):=\widehat{Z}^q_\mathbb{R}(X)/\widehat{\Rat}_{\mathbb R}^q(X)
\]
is called the \emph{$q$-th arithmetic Chow space of~$X$}.
\end{defin}

We can get a feeling about the arithmetic Chow spaces of an arithmetic variety by looking at some simple examples.
In the computations, the following classical result in algebraic number theory provides a very useful tool; we refer the reader for instance to~\cite[Corollary~3.4.7]{Mor13} for its proof.

\begin{thm}[Dirichlet's unit theorem]\label{dirichlet_unit_theorem}
Let $K$ be a number field and $\Sigma$ denote the set of embeddings~$K\hookrightarrow\mathbb{C}$.
Assume that $(c_{\sigma})_{\sigma\in\Sigma}$ is a tuple of real numbers satisfying $c_\sigma=c_{\overline{\sigma}}$ and~$\sum_\sigma c_\sigma=0$.
Then, there exist $\xi_1,\ldots,\xi_s\in O_K^\times$ and $a_1,\ldots,a_s\in\mathbb{R}$ such that
\[
c_\sigma=a_1\log|\sigma(\xi_1)|^2+\ldots+a_s\log|\sigma(\xi_s)|^2
\]
for all~$\sigma\in\Sigma$.
\end{thm}

\begin{exa}\label{CH_0}
For any choice of~$X$, one has that~$\widehat{\CH}^0_{\mathbb{R}}(X)\simeq\mathbb{R}$.
Indeed, as $X$ is assumed to be irreducible, the only arithmetic $\mathbb{R}$-cycles of codimension $0$ are the pairs of the form $(a X,(0)_\sigma)$ for a real number~$a$, and there are no nontrivial elements in~$\widehat{\Rat}_{\mathbb R}^0(X)$.
\end{exa}

\begin{exa}\label{exa_spec_O_K}
Let~$X=\spec O_K$.
Its generic fiber is~$\spec K$, hence for any choice of $\sigma$ the corresponding analytic space consists of a single point.
Smooth functions on a point are identifiable to the set of real numbers, and hence the only currents of bidegree~$(0,0)$ are the real multiples of~$\delta_{X_\sigma}$, which are the currents associated to real constants.
As a result, an arithmetic $\mathbb{R}$-divisor on $X$ is a pair
\[
\widehat D=\left(\sum_{i=1}^r a_i\mathfrak{p}_i, (c_\sigma)_\sigma\right)\,,
\]
where $D=\sum_ia_i\mathfrak{p}_i$ is a real linear combination of nonzero prime ideals of~$O_K$, and $c_\sigma$ is a real number for all~$\sigma$ with the property that~$c_\sigma=c_{\overline{\sigma}}$.
Analogously to~\cite{Ara74}, it has an \emph{Arakelov degree} defined as
\[
\widehat{\deg}\left(\widehat{D}\right)
:=
\sum_{i=1}^ra_i\log\mathfrak N(\mathfrak p_i)+\frac{1}{2}\sum_{\sigma\in\Sigma}c_\sigma\,,
\]
with $\mathfrak{N}(\mathfrak{p}_i)$ being the cardinality of the field~$O_K/\mathfrak{p}_i$.
Since the Picard group of $\spec O_K$ is finite, for all~$i=1,\ldots,r$ there exist $m_i\in\mathbb{N}_{\geq1}$ and $\xi_i\in K^\times$ such that~$\mathfrak{p}_i^{m_i}=(\xi_i)$.
Therefore, $\widehat D$ can be expressed in $\widehat{\CH}^1_\mathbb{R}(X)$ as
\[
\widehat{D}-\sum_{i=1}^r\frac{a_i}{m_i}\widehat{\divi}(\xi_i)=\left(0,(c_\sigma^\prime)_\sigma\right)
\]
for certain real constants $c_\sigma^\prime$ satisfying~$c^\prime_\sigma=c^\prime_{\overline{\sigma}}$.
Therefore, it follows from Dirichlet's unit theorem and from Remark~\ref{rem_div_const_expl} that the map~$\widehat{\deg}\colon\widehat{\CH}^1_\mathbb{R}(\spec O_K)\to\mathbb{R}$, which is well-defined because of the product formula, is an isomorphism of real vector spaces (see also~\cite[Proposition~3.4.5]{Mor13}).
\end{exa}

\begin{rem}
In the classical theory developed in~\cite{GS90a}, arithmetic cycles are defined by only allowing, in the geometric part, $\mathbb{Z}$-linear combinations of $q$-codimensional integral subschemes of~$X$.
In a similar fashion as explained above, this choice induces $q$-th arithmetic Chow groups~$\widehat{\CH}^q(X)$, see \cite[Definition~3.3.4]{GS90a}.
The natural map
\[
\widehat{\CH}^q(X)\longrightarrow\widehat{\CH}_\mathbb{R}^q(X)
\]
is well defined and its image spans~$\widehat{\CH}_\mathbb{R}^q(X)$. 
We underline however that the real vector spaces $\widehat{\CH}^q(X)\otimes_{\mathbb{Z}}\mathbb{R}$ and $\widehat{\CH}^q_{\mathbb{R}}(X)$ are in general not isomorphic, as~$\mathbb{R}\otimes_{\mathbb{Z}}\mathbb{R}\not\simeq\mathbb{R}$.
For instance, compare Example~\ref{exa_spec_O_K} with the fact that the real vector space $\widehat{\CH}^1(\spec\mathbb{Z})\otimes_{\mathbb{Z}}\mathbb{R}$ is infinite dimensional by~\cite[section~3.4.3]{GS90a}.
\end{rem}

We now briefly recall the construction of the arithmetic intersection product between arithmetic $\mathbb{R}$-cycles.
Let $Y_{\sigma}$ be an integral subscheme of $X_\sigma$ of codimension~$q$.
A \emph{Green form of log type for $Y_\sigma$} is a smooth differential form $G_{Y_\sigma}$ of bidegree $(q-1,q-1)$ on the analytic open set $X^{\an}_\sigma\setminus Y^{\an}_\sigma$ such that:
\begin{itemize}
    \item[\sbt] $G_{Y_\sigma}$ is of log type along~$Y^{\an}_\sigma$ (see \cite[1.3.2]{GS90a} for the precise definition of ``log type").
    \item[\sbt]  $[G_{Y_\sigma}]$ is a Green current for~$Y_\sigma$.
\end{itemize}
For a not necessarily prime $\mathbb{R}$-cycle~$Y_\sigma=\sum_i a_i Y_{i,\sigma}$ of pure dimension~$q$, a Green form of log type for $Y_\sigma$ can be defined as
\[
G_{Y_\sigma}:=\sum_i a_i G_{Y_{i,\sigma}}\,,
\]
where each $G_{Y_{i,\sigma}}$ is a Green form of log type for~$Y_{i,\sigma}$.
In \cite[1.3.5, Theorem]{GS90a} it is shown that for a given~$Y$, a Green form of log type for $Y_{\sigma}$ always exists.

Now we are going to define a ``wedge product'' between  two particular currents associated to cycles.
Consider an $\mathbb{R}$-cycle $Y_\sigma$ of pure codimension $q$ on~$X_\sigma$, and let $Z_\sigma\subseteq X_\sigma$ be an integral subscheme of codimension $p$ such that~$Z_\sigma$ is not contained in the support of~$Y_\sigma$.
Moreover consider~$\psi_\sigma:=\iota_\sigma\circ f_\sigma$, where $f_\sigma$ is a desingularization of $Z_\sigma$ and $\iota_\sigma\colon Z_\sigma\hookrightarrow X_\sigma$ is the closed embedding.
If $G_{Y_\sigma}$ is a Green form of log type for~$Y_\sigma$, we put:
\[
[G_{Y_\sigma}]\wedge \delta_{Z_\sigma}:=(\psi_{\sigma})_\ast[\psi_\sigma^\ast G_{Y_\sigma}]\,,
\]
which is a current of bidegree~$(p+q-1,p+q-1)$.
This extends by linearity on $Z_\sigma$ to the case of $\mathbb{R}$-cycles whose irreducible components are not contained in the support of~$Y_\sigma$.

At this point we are ready to give the definition of the \emph{$\ast$-product} between Green currents.
Let $Y$ and $Z$ be two $\mathbb{R}$-cycles of $X$ intersecting properly on~$X_K$.
For~$\sigma\in\Sigma$, consider Green currents $g_{Y,\sigma}$ and $g_{Z,\sigma}$ for $Y_\sigma$ and $Z_\sigma$ respectively.
Then we put:
\begin{equation}\label{star_prod}
g_{Y,\sigma}\ast g_{Z,\sigma}:=[G_{Y_\sigma}]\wedge \delta_{Z_\sigma}+\omega(Y,g_{Y,\sigma})\wedge g_{Z,\sigma}\,,
\end{equation}
where $G_{Y_\sigma}$ is any Green form of log type for $Y_\sigma$ such that~$g_{Y,\sigma}-[G_{Y_\sigma}]\in \im(\partial+\bar{\partial})$ (note that such a $G_{Y_\sigma}$ always exists thanks to~\cite[Theorem~1.3.5, and Lemma~1.2.4]{GS90a}).
One can show that the binary operation defined in \eqref{star_prod} yields a uniquely defined current modulo $\im(\partial+\bar{\partial})$ and moreover that it is associative and commutative modulo~$\im(\partial+\bar{\partial})$.

Given two arithmetic $\mathbb R$-cycles $\widehat Y=(Y,(g_{Y,\sigma})_\sigma)\in \widehat{\CH}_{\mathbb{R}}^q(X)$ and~$\widehat Z=(Z,(g_{Z,\sigma})_\sigma)\in \widehat{\CH}_{\mathbb{R}}^p(X)$, \emph{with $Y$ and $Z$ intersecting properly on~$X$}, the  following intersection product is well defined:
\begin{equation}\label{int_prod}
\widehat Y\cdot\widehat Z:=\left (Y\cdot Z, (g_{Y,\sigma}\ast g_{Z,\sigma})_{\sigma}\right )\in \widehat{\CH}_{\mathbb{R}}^{p+q}(X)\,;
\end{equation}
here, $Y\cdot Z$ is the usual schematic intersection product between $\mathbb{R}$-cycles of~$X$.
More in general, thanks to \cite[Theorems~4.2.3 and~4.5.1]{GS90a} (see also \cite[III.2.2, Theorem~2]{SABK94}) we have a uniquely defined bilinear and symmetric \emph{arithmetic intersection product}:
\[
\widehat{\CH}^p_{\mathbb R}(X)\times\widehat{\CH}^q_{\mathbb R}(X)\longrightarrow\widehat{\CH}^{p+q}_{\mathbb R}(X)
\]
that extends \eqref{int_prod} and makes $\widehat{\CH}_{\mathbb R}^\ast(X)=\bigoplus_{q\ge 0}\widehat{\CH}_{\mathbb R}^q(X)$ a commutative, graded~$\mathbb R$-algebra.
If $\widehat Y$ is an arithmetic $\mathbb{R}$-cycle on~$X$, we naturally have an arithmetic $\mathbb{R}$-cycle $\pi_\ast\widehat Y$ on~$B$, which consists of the proper pushforward of cycles in the geometric part and of currents in the infinite part.
So, for all~$q\in\{0,\ldots,d+1\}$, the composition of the arithmetic intersection product, the operator~$\pi_\ast$, and the Arakelov degree on the base $\widehat{\CH}^1_{\mathbb R}(B)$ defined as in Example~\ref{exa_spec_O_K} give the  bilinear \emph{intersection pairing}
\begin{equation}\label{arithmetic intersection pairing}
\left<\;,\,\right>\colon\widehat{\CH}_{\mathbb{R}}^q(X)\times\widehat{\CH}_{\mathbb{R}}^{d+1-q}(X)\longrightarrow\mathbb{R}\,.
\end{equation}
Using the associativity of the intersection product, for all triples $\widehat{Y},\widehat{Z},\widehat{W}$ of arithmetic $\mathbb{R}$-cycles on $X$ with codimensions summing to~$d+1$, one has
\[
\left<\widehat{Y},\widehat{Z}\cdot \widehat{W}\right>
=
\left<\widehat{Y}\cdot \widehat{Z}, \widehat{W}\right>.
\]

\begin{rem}\label{intersect_vert_fib}
Recall the identification between vertical $\mathbb{R}$-cycles and their associated canonical arithmetic $\mathbb{R}$-cycles of~$X$, as explained in Remark~\ref{canon_image_vert_arithm}.
Let $Z$ be a vertical $\mathbb{R}$-cycle of~$X$, and $\widehat{Y}=(Y,(g_{Y,\sigma})_\sigma)$ an arbitrary arithmetic $\mathbb{R}$-cycle with $Y$ and $Z$ intersecting properly.
It follows from \eqref{int_prod} and the fact that $Z_\sigma=0$ for all $\sigma$ that~$\widehat{Y}\cdot Z=Y\cdot Z$.
\end{rem}

\begin{rem}\label{mor_int}
Let~$\widehat D=(D,(g_\sigma))\in\widehat{\CH}^1_{\mathbb R}(X)$, then without any change in the actual formulas, it is possible to define the intersection pairing $\left<\widehat D,\widehat C\right>$ when $\widehat C=(C, (t_\sigma)_\sigma)$ is just a couple made of a real cycle $C$ of codimension~$d$ and a collection of currents  $(t_\sigma)_\sigma$ of bidegree $(d-1,d-1)$ of real type which are \emph{not} necessarily Green currents  for~$C_\sigma$.
This can be done by using a version of the moving lemma for $\mathbb{R}$-divisors to reduce to the case of proper intersection, then by applying equation \eqref{int_prod} to calculate $\widehat D\cdot\widehat C$ as in~\cite[Definition~5.16]{Mor14}, and finally by taking the pushforward and the Arakelov degree on the base $B$ as explained above.
We observe in particular that, up to moving $D$ so that it intersects $C$ properly, one has
\[
\widehat{\deg}\,\big(\widehat{D}_{|C}\big)
:=
\left<\widehat D, (C,(0)_\sigma)\right>
=\widehat{\deg}\bigg(\pi_*\Big(D\cdot C,([g_\sigma]\wedge\delta_{C,\sigma})_\sigma\Big)\bigg)
\]
is the extension to arithmetic $\mathbb{R}$-divisors of the arithmetic degree of $\widehat{D}$ along $C$ from \cite[section~2.3]{BGS94}.
We refer the reader to \cite[section~5.3]{Mor12} and \cite[sections~0.2 and~2.1]{Mor13} for a definition of $\widehat{\deg}\,\big(\widehat{D}_{|C}\big)$ under more relaxed hypotheses on~$\widehat{D}$, for its extensions to the intersection numbers of higher dimensional cycles and for its comparison with analogous generalizations by S.-W. Zhang in \cite[Lemma~6.5]{Zha95} and by V. Maillot in \cite[chapitre~5]{Mai00}.

Nevertheless, we underline the fact that in the more general assumptions of this remark, the object~$\widehat D\cdot\widehat C$, although well-defined, is not necessarily an element of~$\widehat{\CH}^{d+1}_{\mathbb R}(X)$.
\end{rem}

\subsection{Arithmetic Hodge index theorems}

We now introduce the higher dimensional arithmetic Hodge index theorem for arithmetic divisors and $\mathbb R$-divisors on an arithmetic variety $X$ which satisfies the same properties as at the beginning of the section; their proof is due to A. Moriwaki.
For simplicity we adopt the following notations:
\begin{itemize}
    \item[\sbt] If $Z$ is a $\mathbb{R}$-cycle on~$X$, then $(Z_K)^m$ denotes the $m$-fold intersection product of the restriction of $Z$ on the $K$-algebraic variety~$X_K$.
    \item[\sbt] If $\widehat Z$ is an arithmetic $\mathbb R$-cycle on~$X$, then $\widehat Z^m$ denotes the $m$-fold arithmetic intersection product.
\end{itemize}

Let $\widehat A=(A,(g_\sigma)_\sigma)\in\widehat{\CH}^1(X)$; because of~\cite[Proposition~2.5.(iv)]{GS90b}, it can be expressed as $(\divi(s),(-\log h_\sigma(s,s))_\sigma)$ for a rational section $s$ of a unique (up to isometry) hermitian line bundle $(\mathscr L,(h_\sigma)_\sigma)$ on~$X$.
We say that $\widehat{A}$ is \emph{ample} if $\mathscr L$ is (relatively) ample on~$X$, the form $\omega(A,g_\sigma)$ is positive definite for all $\sigma$ and $H^0(X,n\mathscr L)$ is generated by strictly small sections as a $\mathbb{Z}$-module for $n$ big enough (see \cite[section~5.10]{Mor14} for details).

An arithmetic $\mathbb R$-divisor $\widehat D=(D,(g_\sigma)_\sigma)$ is said to be \emph{nef} if $\widehat{\deg}\,\big(\widehat{D}_{|C}\big)\ge 0$ for any integral curve~$C\subseteq X$ (recall Remark~\ref{mor_int}) and $\omega(D,g_\sigma)$ is positive semi-definite for all~$\sigma$. 
Finally, we say that $\widehat D$ is \emph{integrable} (in the sense of Moriwaki) if it can be expressed as the difference of two nef arithmetic $\mathbb R$-divisors.

\begin{thm}[arithmetic Hodge index theorem]\label{ZHod}
Assume that~$d\geq1$, and fix an ample arithmetic divisor $\widehat  A$ on~$X$.
For any arithmetic divisor $\widehat D$ such that
\[
\deg\left(D_K\cdot(A_K)^{d-1}\right)=0
\]
we have that:
\begin{equation}\label{self_neg}
\left<\widehat D^2,\widehat A^{d-1}\right>\le 0\,.
\end{equation}
The equality in \eqref{self_neg} is achieved if and only if there are $\widehat{E}\in \widehat{\CH}^1(B)$ and $n\in\mathbb N_{\geq1}$ such that~$n\widehat D=\pi^{\ast}(\widehat E)$.
\end{thm}
\proof
See \cite[Theorem~B]{Mor96}.
\endproof

The extension of Theorem~\ref{ZHod} is not straightforward for arithmetic~$\mathbb R$-divisors, since a negative definite quadratic form $q\colon\mathbb Q^2\to\mathbb{R}$ doesn't necessarily extend to a negative definite quadratic form on~$\mathbb R^2$. 
Nevertheless we have the following weaker result:

\begin{thm}[arithmetic Hodge index theorem for real divisors]\label{RHod}
Assume that~$d\geq1$, and fix an ample arithmetic divisor $\widehat  A$ on~$X$.
For any integrable arithmetic $\mathbb{R}$-divisor $\widehat D$ such that
\[
\deg\left(D_K\cdot(A_K)^{d-1}\right)=0
\]
we have that:
\begin{equation}\label{rself_neg}
\left<\widehat D^2,\widehat A^{d-1}\right>\le 0\,.
\end{equation}
Moreover, if in \eqref{rself_neg} equality holds, then $D_K=0$ in~$\CH^1_{\mathbb R}(X_K)$.
\end{thm}
\proof
See \cite[Theorem~2.2.5]{Mor13}.
\endproof

\begin{rem}
The statement of \cite[Theorem~2.2.5]{Mor13} holds in a more general setting: in fact the author allows Green currents to be just of~$\mathscr C^0$-type, and $\widehat A$ to be an ample arithmetic~$\mathbb Q$-divisor.
Although the original proof is given for the case~$K=\mathbb Q$, it is not difficult to modify it for a general number field.
\end{rem}

We conclude this subsection by giving a form of the Hodge index theorem for arithmetic $\mathbb R$-divisors which are vertical over a prime of the basis.

To state it, recall that for a prime $\mathfrak{p}\neq(0)$ we denote by $Z^1_{\mathbb R}(X)_{\mathfrak p}$ the space of $\mathbb R$-divisors of $X$ with support contained in the fiber~$X_\mathfrak p$.
As in Remark~\ref{canon_image_vert_arithm}, we identify an $\mathbb{R}$-divisor $D$ of $X$ which is vertical over $\mathfrak{p}$ with its canonical image in~$\widehat{Z}^1_{\mathbb R}(X)$.

\begin{prop}\label{hi_li}
Fix an ample arithmetic divisor $\widehat{A}$ on~$X$, and a prime $\mathfrak p\neq(0)$ in~$B$.
Then, the association
\begin{eqnarray*}
\left<\;,\,\right>_{\mathfrak p,\widehat {A}}\colon Z^1_{\mathbb R}(X)_{\mathfrak p}\times Z^1_{\mathbb R}(X)_{\mathfrak p}&\longrightarrow& \mathbb R\\
(D,E) &\longmapsto& \left<D,E\cdot\widehat {A}^{d-1}\right>
\end{eqnarray*}
is a negative semi-definite symmetric bilinear form, and $\left<D,D\right>_{\mathfrak p,\widehat {A}}=0$ if and only if~$D\in\mathbb RX_{\mathfrak p}$.
\end{prop}
\proof
The function $\left<\;,\,\right>_{\mathfrak p,\widehat {A}}$ is clearly bilinear and symmetric, because of the corresponding properties of the arithmetic intersection product.
Therefore, we can prove the claim by means of the Zariski's Lemma for general finite dimensional real vector spaces, for instance in the form of \cite[Lemma~1.1.4]{Mor13}.
As $\widehat A$ is fixed, we will put for simplicity of notations $\left<\;,\,\right>_{\mathfrak p}:=\left<\;,\,\right>_{\mathfrak p,\widehat {A}}$ in the remaining of the proof, and denote by $A$ the corresponding underlying divisor on~$X$.

The irreducible components $\Gamma_1,\ldots,\Gamma_r$ of the fiber of $X$ over~$\mathfrak{p}$ are generators for the real vector space~$Z_{\mathbb{R}}^1(X)_\mathfrak{p}$, and we have that $X_{\mathfrak p}=\sum_i a_i\Gamma_i$ for some positive integers~$a_1,\ldots,a_r$.

As $\CH^1(B)$ is a finite group, the geometric cycle $X_{\mathfrak{p}}=\pi^\ast\mathfrak{p}$ is a torsion element in~$\CH^1(X)$.
Therefore, there exist $m\in\mathbb{N}_{\geq1}$ and $f\in K(X)^\times$ such that~$m X_{\mathfrak p}=\divi(f)$.
This implies that the equality $(X_\mathfrak{p},(0)_\sigma)=(0,(\log|f_\sigma|^2/m)_\sigma)$ holds in~$\widehat{\CH}_{\mathbb{R}}^1(X)$ and then that
\[
\left<\Gamma_i,X_{\mathfrak p}\right>_{\mathfrak p}
=
\left<\Gamma_i,(0,(\log|f_\sigma|^2/m)_\sigma)\cdot\widehat{A}^{d-1}\right>
=
0
\]
for all~$i=1,\ldots,r$, thanks to Remark~\ref{intersect_vert_fib}.

The same remark ensures that for all $i\neq j$ one has~$\left<\Gamma_i,\Gamma_j\right>_{\mathfrak p}=\log \mathfrak N(\mathfrak{p})\deg\left(\Gamma_i\cdot\Gamma_j\cdot A^{d-1}\right)$.
The product $\Gamma_i\cdot\Gamma_j$ is effective by definition, see for instance \cite[section~2.3]{Ful98}.
As $A$ is ample, the Nakai-Moishezon criterion gives that $\left<\Gamma_i,\Gamma_j\right>_{\mathfrak p}\geq0$ for all choice of~$i\neq j$, and moreover the sign is strictly positive if $\Gamma_i$ and $\Gamma_j$ have nonempty intersection, refer to \cite[Lemma~12.1]{Ful98}.

Finally, recall that the fiber of $X$ over $\mathfrak{p}$ is connected.
Hence, for every $i\neq j$ there exist a sequence of indices $i=i_1,i_2,\ldots,i_\ell=j$ such that $\Gamma_k$ and $\Gamma_{k+1}$ have nonempty intersection for all~$k=1,\ldots,\ell-1$.

Applying readily \cite[Lemma~1.1.4 (2)]{Mor13}, one concludes the proof of the statement. 
\endproof

Notice that related proofs of the previous result in slightly different settings can be found in \cite[Lemma~2.1]{HPW05} and (in the case of surfaces) in \cite[Theorem~9.1.23]{Liu02}.

\section{Characterization of numerically trivial arithmetic \texorpdfstring{$\mathbb{R}$}{R}-divisors}

This section is devoted to prove that, for arithmetic~$\mathbb R$-divisors, numerical equivalence is actually the same as rational equivalence.
This result will be an extension and a generalization of \cite[Theorem~7.1]{BC09}, of which we follow the strategy.
Note that, for algebraic varieties over a field, numerically equivalent divisors are in general not rationally equivalent; therefore we are going to prove an important distinctive feature of arithmetic geometry.

\begin{defin}\label{num-Pic}
An  arithmetic $\mathbb R$-divisor $\widehat{D}$ is said to be \emph{numerically trivial} if 
\[
\left<\widehat{D},\widehat C\right>=0
\]
for all~$\widehat C\in\widehat{\CH}_{\mathbb R}^d(X)$.
The set of all numerically trivial arithmetic $\mathbb R$-divisors is a real vector subspace of~$\widehat{Z}^1_{\mathbb{R}}(X)$, denoted by~$\widehat{\Num}^{\tr}_{\mathbb R}(X)$.
Two arithmetic $\mathbb{R}$-divisors are said to be \emph{numerically equivalent} if their difference is numerically trivial.
\end{defin}

Before considering the general situation, it is illustrative to consider the simpler case of~$d=0$, and characterize numerically trivial arithmetic~$\mathbb{R}$-divisors in this instance.

\begin{exa}\label{num_tri_on_O_K}
Let~$X=\spec O_K$.
Because of Example~\ref{CH_0}, an arithmetic $\mathbb{R}$-divisor $\widehat{D}$ on $X$ is numerically trivial if and only if it has vanishing Arakelov degree.
As in Example~\ref{exa_spec_O_K}, the Dirichlet's unit theorem implies then that $\widehat{D}$ is a real linear combination of principal arithmetic divisors.
\end{exa}

Even in the general setting, numerically trivial arithmetic $\mathbb{R}$-divisors satisfy several restrictive properties.
The first one concerns their curvature form.

\begin{prop}\label{omega_is_zero}
If $\widehat{D}=(D,(g_\sigma)_\sigma)$ is a numerically trivial arithmetic~$\mathbb{R}$-divisor, then $\omega(D,g_\sigma)=0$ for all~$\sigma\in\Sigma$.
\end{prop}
\begin{proof}
Assume that~$d\geq1$, as the claim is otherwise trivial, and fix~$\tau\in\Sigma$.
The complex vector space $\mathscr{A}^{1,1}(X^{\an}_\tau)$ of smooth forms on $X^{\an}_\tau$ of bidegree $(1,1)$ is a locally convex topological space when equipped with its Schwartz's topology from \cite[section~9]{deR84}.

Assuming by contradiction that~$\omega(D,g_\tau)\neq0$, one can consider the continuous linear functional on the linear span of $\omega(D,g_\tau)$ of $\mathscr{A}^{1,1}(X^{\an}_\tau)$ defined by sending $\omega(D,g_\tau)$ to~$1$.
The Hahn-Banach theorem for locally convex topological vector spaces (see \cite[Theorem~3.6]{Rud91}) ensures then the existence of a continuous linear functional $T$ on $\mathscr{A}^{1,1}(X^{\an}_\tau)$ such that~$T(\omega(D,g_\tau))=1$.
By definition, $T$ is a current of bidegree~$(d-1,d-1)$.

Because of~\cite[Theorem~12 at page~68]{deR84}, the set of smooth forms of bidegree~$(d-1,d-1)$ has dense image in the space of currents of the same bidegree.
In particular, $T$ can be approximated by a sequence $(t_k)_k$ of Green currents for the $d$-codimensional zero cycle on~$X^{\an}_\tau$, and so
\[
1=T(\omega(D,g_\tau))=\lim_k\,t_k(\omega(D,g_\tau))\,.
\]
To get a contradiction, we show that for all choice of a Green current $t$ for the $d$-codimensional zero cycle on~$X^{\an}_\tau$, one must have~$t(\omega(D,g_\tau))=0$.
To do this, assume that $\tau$ is a complex embedding, the proof being analogous if $\tau$ is real, and consider the arithmetic cycle of codimension~$d$:
\[
t X_\tau:=(0, (t_\sigma)_\sigma)\quad\text{with } t_\sigma=0 \text{ for } \sigma\notin\{\tau,\overline{\tau}\},\, t_\tau=t\text{ and }t_{\overline{\tau}}=(-1)^{d-1}F_\tau^\ast t\,.
\]
Since $\widehat D$ is numerically trivial, the explicit expression of $\left<\widehat D, tX_\tau\right>=0$ implies that~$(\omega(D,g_\tau)\wedge t)(1)=0$, as desired.
\end{proof}

Before listing other features of numerically trivial arithmetic~$\mathbb{R}$-divisors, we recall that an $\mathbb{R}$-divisor $D$ on $X$ is said to be \emph{divisorially $\pi$-numerically trivial} if $\deg(D_{|C})=0$ for any curve $C$ contained in the support of some fiber~$X_{\mathfrak p}$; here $\deg(D_{|C})$ denotes the usual geometric degree of $D$ along the curve~$C$.

\begin{lem}\label{forRHod}
Let $\widehat{D}=(D,(g_\sigma)_\sigma)$ be a numerically trivial arithmetic $\mathbb{R}$-divisor of~$X$.
Then:
\begin{enumerate}[{(1)}]
    \item $D$ is divisorially $\pi$-numerically trivial.
    \item $\deg(D_K\cdot\mathcal C)=0$ for any integral curve~$\mathcal C\subseteq X_K$.
    \item $\widehat D$ is nef; in particular, it is integrable in the sense of Moriwaki.
\end{enumerate}
\end{lem}
\proof
For the first claim, let $C$ be any vertical curve~$C\subseteq X_{\mathfrak p}$, and as usual identify it with an arithmetic~$\mathbb{R}$-cycle.
Because of Remark~\ref{intersect_vert_fib}, one has that
\[
0
=
\left<\widehat{D},C\right>
=
\log\mathfrak N(\mathfrak p)\deg(D_{|C})\,,
\]
implying the conclusion.

For the proof of~(2), it is enough to recall that the degree of a $0$-cycle is invariant under base change, and that for any choice of $\sigma\in\Sigma$ one has that~$\omega(D,g_\sigma)=c_1(\overline{\mathscr{O}(D_\sigma)})$, where the line bundle $\mathscr O(D_\sigma)$ is equipped with the canonical hermitian metric defined from~$g_\sigma$.
Therefore:
\[
\deg(D_K\cdot\mathcal C)
=\deg(D_\sigma\cdot\mathcal C_\sigma)
=\int_{C^{\an}_\sigma}\omega(D,g_\sigma)\,,
\]
but the last integral vanishes because of Proposition~\ref{omega_is_zero}.

In order to prove the third claim, it is enough to show that $\widehat{\deg}\,\big(\widehat{D}_{|C}\big)=0$ for any integral curve $C$ of~$X$, since we already know by Proposition~\ref{omega_is_zero} that  $\omega(D,g_\sigma)=0$ for all~$\sigma$.
Choosing any Green current $t_\sigma$ for~$C_\sigma$, we have by bilinearity of the generalized intersection pairing that
\[
\widehat{\deg}\,\big(\widehat{D}_{|C}\big)
=
\left<\widehat{D},(C,(0)_\sigma)\right>
=
\left<\widehat{D},(C,(t_\sigma)_\sigma)\right>-\left<\widehat{D}, (0,(t_\sigma)_\sigma)\right>
=
-\left<\widehat{D}, (0,(t_\sigma)_\sigma)\right>
\]
as the first summand coincides with the intersection pairing for arithmetic~$\mathbb R$-cycles, and hence it vanishes by the assumption on~$\widehat{D}$.
Recalling that the generalized arithmetic intersection product is defined as in Remark~\ref{mor_int}, we get
\[
\widehat{D}\cdot(0,(t_\sigma)_\sigma)=\left (0, (\omega(D,g_\sigma)\wedge t_\sigma)_\sigma\right)\,.
\]
But in Proposition~\ref{omega_is_zero} we showed that $\omega(D,g_\sigma)=0$ for all~$\sigma$, therefore the claim follows.
\endproof

We finally recall a very useful lemma from commutative algebra.
Here and in the following, we silently identify an element of a $\mathbb{Z}$-module $M$ with its canonical image in~$M\otimes_{\mathbb{Z}}\mathbb{R}$.

\begin{lem}\label{alg_lem}
Let $M$ be a~$\mathbb Z$-module.
Then the following assertions hold:
\begin{enumerate}[{(1)}]
    \item For~$x\in M\otimes_{\mathbb Z} \mathbb R$, there are $x_1,\ldots x_r\in M$ and $\lambda_1,\ldots, \lambda_r\in\mathbb R$ such that $\lambda_1,\ldots, \lambda_r$ are linearly independent over $\mathbb Q$ and~$x=\sum^r_{i=1}\lambda_ix_i$.
    \item Let $x_1,\ldots x_r\in M$ and $\lambda_1,\ldots, \lambda_r\in\mathbb R$ such that $\lambda_1,\ldots, \lambda_r$ are linearly independent over~$\mathbb Q$.
    If $\sum^r_{i=1}\lambda_ix_i=0$ in~$M\otimes_{\mathbb Z} \mathbb R$, then $x_1,\ldots x_r$ are torsion elements of~$M$.
\end{enumerate}
\end{lem}
\proof
See \cite[Lemma~1.1.1]{Mor13}.
\endproof

With the previous results in mind, we dispose of all the ingredients needed to state and prove the main result of this section.

\begin{thm}\label{hi_bo}
The following equality of $\mathbb R$-vector spaces holds:
\[
\widehat{\Num}^{\tr}_{\mathbb R}(X)=\widehat\Rat_{\mathbb R}^1(X)\,.
\]
In other words, two arithmetic $\mathbb{R}$-divisors are numerically equivalent if and only if they are equal in~$\widehat{\CH}^1_{\mathbb R}(X)$. 
\end{thm}
\proof
Clearly~$\widehat\Rat_{\mathbb R}^1(X)\subseteq \widehat{\Num}^{\tr}_{\mathbb R}(X)$, so let us show the other inclusion.
When~$d=0$, the statement is proven in Example~\ref{num_tri_on_O_K}, therefore we assume for the rest of the proof that~$d\geq1$, so that one can apply the various versions of the arithmetic Hodge index theorem.
To do that, we also fix the choice of an ample arithmetic divisor $\widehat{A}$ on~$X$, whose existence is granted by the projectivity assumption on~$X$, see for instance~\cite[Corollary~2.11]{Cha17}.

Our first goal is to show that inside $\widehat{\CH}^1_{\mathbb R}(X)$ a numerically trivial $\mathbb{R}$-divisor $\widehat D=(D,(g_\sigma)_\sigma)$ is the pull-back of an arithmetic $\mathbb{R}$-divisor on~$B$.

Thanks to Lemma~\ref{alg_lem}.(1) we can find $\lambda_1,\ldots\lambda_r\in\mathbb R$ linearly independent over $\mathbb Q$ and $D_1,\ldots, D_r\in \Div(X)$ such that~$D=\sum_i\lambda_iD_i$.
We notice that every $D_i$ is divisorially $\pi$-numerically trivial.
Indeed, for any vertical curve~$C$, one has by Lemma~\ref{forRHod}.(1) that
\[
0=\deg (D_{|C})
=
\sum_{i=1}^r\lambda_i\deg({D_i}_{|C})\,.
\]
So, by the linear independence of the coefficients $\lambda_i$ over $\mathbb Q$ we conclude that $\deg({D_i}_{|C})=0$ for all~$i=1,\ldots,r$.
The analogous argument applied to the generic fiber, together with Lemma~\ref{forRHod}.(2), shows that for any integral curve $\mathcal C\subseteq X_K$ we have~$\deg(D_{i,K}\cdot\mathcal C)=0$.
This last observation implies that, for any embedding~$\sigma$, the Chern class $c_1(\mathscr{O}(D_{i,\sigma}))$ is~$0$.
Therefore, we can choose a Green current $g_{i,\sigma}$ for $D_{i,\sigma}$ such that~$\omega(D_i, g_{i,\sigma})=0$, see~\cite[3.3.5 Theorem,~(ii)]{GS90a}.

Since we also have $\omega(D, g_\sigma)=0$ for every $\sigma$ thanks to Proposition~\ref{omega_is_zero}, we deduce that the current
\[
G_\sigma:=g_\sigma-\sum_{i=1}^r\lambda_i g_{i,\sigma}
\]
satisfies~$dd^c(G_\sigma)=0$.
Now put for example~$g^\prime_{1,\sigma}:=g_{1,\sigma}+\frac{G_\sigma}{\lambda_1}$.
Then, $g^\prime_{i,\sigma}$ is again a Green current for $D_{i,\sigma}$ satisfying $\omega(D_1,g^\prime_{1,\sigma})=0$ and by construction:
\[
\lambda_1 g'_{1,\sigma}+\sum^{r}_{i=2}\lambda_i g_{i,\sigma}=g_\sigma\,.
\]
In other words, up to replacing  $g_{1,\sigma}$ by~$g^\prime_{1,\sigma}$, we have proved that there exists a decomposition
\[
\widehat{D}=\sum_{i=1}^r\lambda_i\widehat{D}_i\,,
\]
where for all $i=1,\ldots,r$ the arithmetic divisor $\widehat{D}_i=(D_i,(g_{i,\sigma})_\sigma)$ is such that $D_i\in\Div(X)$ is divisorially $\pi$-numerically trivial and $\omega(D_i,g_{i,\sigma})=0$ for all~$\sigma$.

Using the entries (2) and (3) of Lemma~\ref{forRHod} and the fact that $\widehat{D}$ is numerically trivial, the arithmetic Hodge index theorem for real divisors stated in Theorem~\ref{RHod} implies that
\[
\sum_{i=1}^r\lambda_i\, D_{i,K}=D_K=0
\quad \textrm{inside  } {\CH}^1_{\mathbb R}(X_K)\,.
\]
Now it follows from Lemma \ref{alg_lem}.(2) that each $D_{i,K}$ is a torsion element in~${\CH}^1(X_K)$, which in turn means that for some~$n_i\in\mathbb{N}_{\geq1}$ the divisor $n_iD_i$ is (up to linear equivalence) vertical and divisorially $\pi$-numerically trivial.
In particular, for the fixed ample arithmetic divisor $\widehat A$ on~$X$, we have that $(n_i\widehat{D}_i)\cdot \widehat{A}^{d-1}$ can be represented by a vertical arithmetic cycle $(C,(t_\sigma)_\sigma)$ of codimension~$d$.
Therefore,
\[
\deg\left((n_iD_i)_K\cdot(A_K)^{d-1}\right)=0
\]
and
\[
\left<(n_i \widehat{D}_i)^2,\widehat A^{d-1}\right>
=
\left<n_i \widehat{D}_i,(C,(t_\sigma)_\sigma)\right>
=
\left<n_i \widehat{D}_i,(C,(0)_\sigma)\right>+n_i\left<\widehat{D}_i,(0,(t_\sigma)_\sigma)\right>
=
0\,,
\]
where the last equality descends from Remark~\ref{intersect_vert_fib}, the fact that $n_iD_i$ is divisorially $\pi$-numerically trivial and that $\omega(D_i,g_{i,\sigma})=0$ for all~$\sigma$.
Thanks to the arithmetic Hodge index theorem in Theorem~\ref{ZHod}, we conclude that there exist a multiple $m_i$ of $n_i$ and an element $\widehat{E}_i\in \widehat{\CH}^1(B)$ such that~$m_i\widehat{D}_i=\pi^{\ast}(\widehat{E}_i)$.
Therefore, as the same equality holds in~$\widehat{\CH}_{\mathbb{R}}^1(X)$, the linearity of the pull-back~$\pi^\ast\colon\widehat{\CH}_{\mathbb{R}}^1(B)\to\widehat{\CH}_{\mathbb{R}}^1(X)$ gives
\[
\widehat{D}=\sum_{i=1}^r\lambda_i\widehat{D}_i=\sum_{i=1}^r\frac{\lambda_i}{m_i}\,\pi^\ast(\widehat{E}_i)=\pi^\ast(\widehat{E})\,,
\]
for~$\widehat{E}:=\sum_{i=1}^r(\lambda_i/m_i)\widehat{E}_i$.

To conclude, it is hence sufficient to show that $\widehat{E}$ vanishes in the first arithmetic Chow space of~$B$.
We know by Example~\ref{exa_spec_O_K} that $\widehat{E}$ can be written in $\widehat{\CH}^1_\mathbb{R}(B)$ in the form~$(0,(c_\sigma)_\sigma)$, where $c_\sigma$ are real constants satisfying~$c_{\sigma}=c_{\overline{\sigma}}$, and that $\widehat{E}=0$ if and only if~$\sum_\sigma c_\sigma=0$.
To check this last equality, it is enough to intersect the numerically trivial arithmetic $\mathbb{R}$-divisor $\widehat{D}=\pi^{\ast}(\widehat{E})$ with an arbitrary arithmetic cycle $\widehat{C}$ of codimension $d$ on $X$ satisfying~$\deg(C_K)\neq 0$.
Indeed, the relation in \cite[Theorem~4.3.9]{GS90a} and the definitions of the intersection product in \eqref{int_prod} and of the Arakelov degree in Example~\ref{exa_spec_O_K} yield
\[
0
=
\left<\widehat{D},\widehat{C}\right>
=
\left<\pi^{\ast}(\widehat{E}),\widehat{C}\right>
=
\widehat{\deg}\left(\widehat{E}\cdot\pi_{\ast}(\widehat{C})\right)
=
\frac{\deg(C_K)}{2}\sum_{\sigma\in\Sigma}c_\sigma\,,
\]
as desired.
\endproof

One interpretation of the result in Theorem~\ref{hi_bo} is that the numerical equivalence of arithmetic $\mathbb{R}$-divisors is a very restraining relation, preventing the novelty of a natural definition for the numerical Chow real vector space in the context of arithmetic geometry.

It also asserts that the arithmetic intersection pairing in \eqref{arithmetic intersection pairing} is non-degenerate in the first argument when~$q=1$, proving part of a conjecture of H. Gillet and C. Soul\'e, see \cite[Conjecture~1]{GS94}.
Notice that Theorem~\ref{hi_bo} together with one part of the arithmetic hard Lefschetz conjecture in \cite[Conjecture~2.(i)]{GS94} would be sufficient to confirm the full non-degeneracy for~$q=1$.
To summarize, we have the following:

\begin{cor}
Let $X$ be a projective arithmetic variety of relative dimension $d\ge1$ over the ring of integers of a number field.
Then, the pairing
\[
\widehat{\CH}^1_{\mathbb R}(X)\times\widehat{\CH}^d_{\mathbb R}(X)\longrightarrow\mathbb{R}
\]
is non-degenerate in the first argument.
Moreover, it is non-degenerate if the Lefschetz operator
\[
L^{d-1}\colon \widehat{\CH}^1_{\mathbb R}(X)\to \widehat{\CH}^d_{\mathbb R}(X)
\]
associated to an ample arithmetic divisor on $X$ is surjective.
\end{cor}


\begin{rem}
The general form of Gillet--Soul\'e's conjecture predicts that the pairing \eqref{arithmetic intersection pairing} is non-degenerate for all~$q\in\{0,\ldots,d+1\}$ and would be implied by the verification of arithmetic analogues of Grothendieck's standard conjectures, see \cite[section~1]{GS94} for details.

Regarding these last ones, we recall that it is sufficient to check them for Arakelov cycles, as proved in \cite[Proposition~3.1]{Kue95}.
Moreover, they hold true when $d=1$ (thanks to the Hodge index theorem for arithmetic surfaces proved by G. Faltings and P. Hriljac) or~$q=0$ (by \cite[Theorem~1.(i)]{GS94}).
Also, K. K{\"u}nnemann in~\cite{Kue95}, K. K{\"u}nnemann and V. Maillot in \cite{KM00} and Y. Takeda in \cite{Tak98} proved that both conjectures holds true for some special arithmetic varieties, including projective spaces.
A. Moriwaki gave a partial answer to the second one when $q=1$ in \cite[Theorem~A]{Mor96} as a consequence of the higher dimensional arithmetic Hodge index theorem.
\end{rem}

\section{The first Arakelov-Chow space and its dimension}

As in the previous sections, we consider a number field $K$ with ring of integers $O_K$ and set of embeddings $\Sigma$ of $K$ into~$\mathbb{C}$, and a projective arithmetic variety $X$ of relative dimension $d$ over~$O_K$.
Moreover, for all~$\sigma\in\Sigma$, we fix the choice of a K\"ahler form $\Omega_\sigma$ on $X^{\an}_\sigma$ such that $F_\sigma^\ast(\Omega_\sigma)=-\Omega_{\overline{\sigma}}$ and we put~$\Omega=(\Omega_\sigma)_{\sigma\in\Sigma}$.
Notice that, for all embedding~$\sigma$, the choice of $\Omega_\sigma$ defines a $\bar{\partial}$-Laplace operator on the space of smooth forms on~$X^{\an}_\sigma$, and then a notion of harmonicity.

We take advantage of this extra analytic structure to recall the definition of Arakelov $\mathbb{R}$-cycles and to attach to $X$ its first Arakelov-Chow space.
We then show that such a real vector space is finite dimensional and we express its dimension in terms of the geometry and arithmetic of~$X$.

\subsection{Arakelov \texorpdfstring{$\mathbb{R}$}{R}-divisors}

It is not difficult to see that, whenever~$d\geq1$, the real vector space of arithmetic $\mathbb{R}$-divisors modulo rational equivalence has infinite dimension.
For this reason it is useful, following \cite[section~5.1]{GS90a}, to restrict our attention to a particular type of arithmetic~$\mathbb R$-divisors, which turn out to be precisely those originally introduced by Arakelov in \cite{Ara74} when~$d=1$.
The definition can in fact be given for arithmetic $\mathbb{R}$-cycles of arbitrary codimension~$q\in\{0,\ldots,d+1\}$; to do this, we recall that if $(Y,(g_\sigma)_\sigma)$ is an arithmetic $\mathbb{R}$-cycle of codimension $q$ there exists, for all embedding~$\sigma$, a smooth form $\omega(Y,g_\sigma)$ of bidegree $(q,q)$ on $X_\sigma^{\an}$ for which~$dd^cg_\sigma+\delta_{Y_\sigma}=[\omega(Y,g_\sigma)]$.

\begin{defin}\label{defn Arakelov cycles}
An \emph{Arakelov $\mathbb{R}$-cycle of codimension $q$ on $X$} is an element $(Y,(g_\sigma)_\sigma)$ of~$\widehat{Z}^q_\mathbb{R}(X)$ such that $\omega(Y,g_\sigma)$ is harmonic with respect to~$\Omega_\sigma$, for all~$\sigma\in\Sigma$.
Arakelov $\mathbb{R}$-cycles of codimension $q$ on $X$ form a real vector subspace, which is denoted by~$\overline{Z}^q_\mathbb{R}(X)$.
An element of~$\overline{Z}^1_\mathbb{R}(X)$ is simply called an \emph{Arakelov~$\mathbb R$-divisor}.
\end{defin}

Thanks to the Lelong-Poincar\'e formula, the arithmetic $\mathbb{R}$-cycles in $\widehat{\Rat}_{\mathbb R}^q(X)$ are actually Arakelov~$\mathbb{R}$-cycles, therefore we can define the analogue of the arithmetic Chow spaces of $X$ in this setting.

\begin{defin}
The \emph{$q$-th Arakelov-Chow space of $X$} is defined as the quotient
\[
\overline{\CH}^q_{\mathbb R}(X)
:=
\overline{Z}^q_\mathbb{R}(X)/\widehat{\Rat}_{\mathbb R}^q(X)\,,
\]
or equivalently as the image of $\overline{Z}_{\mathbb{R}}^q(X)$ in~$\widehat{\CH}^q_{\mathbb R}(X)$.
\end{defin}

Notice that, for simplicity of notations, we dropped the dependence from the choice of the K\"ahler form $\Omega$ in the definitions of ~$\overline{Z}^q_\mathbb{R}(X)$ and~$\overline{\CH}^q_{\mathbb R}(X)$, as we suppose it fixed throughout all the section.

In the easiest cases, the Arakelov-Chow spaces match the arithmetic Chow spaces introduced in Definition~\ref{def_arith_CH}, as the following two examples show.

\begin{exa}
For any choice of~$X$, the zero codimensional arithmetic cycle $(X,(0)_\sigma)$ is of Arakelov type, as~$\omega(X,0)=1$.
In particular, as in Example~\ref{CH_0} one has that~$\overline{\CH}^0_{\mathbb R}(X)\simeq\mathbb{R}$.
\end{exa}

\begin{exa}\label{exa_ar_spec_O_K}
Let~$X=\spec O_K$.
Any arithmetic $\mathbb{R}$-divisor on $X$ is also an Arakelov~$\mathbb{R}$-divisor, since $X_\sigma^{\an}$ is zero dimensional.
Therefore, it follows from Example~\ref{exa_spec_O_K} that $\overline{\CH}^1_{\mathbb R}(\spec O_K)$ is a real vector space of dimension~$1$.
\end{exa}

Being an Arakelov $\mathbb{R}$-divisor becomes a more restrictive requirement when~$d\geq1$.
The following well known proposition gives a beautiful and simple way of expressing Arakelov~$\mathbb R$-divisors.
To state it, we denote by $\mathbb R X_\sigma$ the one-dimensional $\mathbb R$-vector space formally generated by the symbol~$X_\sigma$.
There is a natural involution $F_\infty$ on $\bigoplus_{\sigma\in\Sigma}\mathbb R X_\sigma$ obtained by exchanging the real coefficients of $X_\sigma$ and $X_{\overline{\sigma}}$ for all~$\sigma$; the subset of elements of the direct sum which are fixed by this involution is denoted, as usual, by the use of a superscript.

\begin{prop}\label{classical_ar}
The following map between $\mathbb R$-vector spaces is an isomorphism:
\begin{eqnarray*}
\overline{Z}^1_\mathbb{R}(X)&\longrightarrow& Z^1_{\mathbb R}(X)\oplus\left(\bigoplus_{\sigma\in\Sigma}\mathbb R X_\sigma\right)^{F_\infty}
\\
(D,(g_\sigma)_\sigma)&\longmapsto& D+\sum_{\sigma\in\Sigma} H_\sigma(g_\sigma) X_\sigma\,,
\end{eqnarray*}
where $H_\sigma$ is the orthogonal projection onto the space of harmonic functions on~$X^{\an}_\sigma$.
\end{prop}
\proof
First of all, remark that the notion of arithmetic $\mathbb{R}$-divisor from Definition~\ref{defn arithmetic cycles} agrees with the one given in \cite[section~3.3.3]{GS90a}, since in this case the images of $\partial$ and of $\bar\partial$ in the space of currents of bidegree $(0,0)$ are trivial.
Therefore, Arakelov $\mathbb{R}$-divisors according to Definition \ref{defn Arakelov cycles} coincide with the ones considered in the cited reference, and the claim is a real-coefficients version of \cite[section~5.1.1, Lemma]{GS90a}.
Notice in particular that for all $\sigma$ one must have~$H_\sigma(g_\sigma)\in\mathbb{R}$, since $X^{\an}_\sigma$ is a connected compact complex manifold.
\endproof

If an Arakelov $\mathbb R$-divisor is identified with an element of $Z^1_{\mathbb R}(X)\oplus\left(\bigoplus_\sigma\mathbb R X_\sigma\right)^{F_\infty}$ via the isomorphism of the above proposition, we say that it is expressed in \emph{explicit form}.

\begin{nota}
From now on, we will freely use the identification given by the statement of Proposition~\ref{classical_ar}.
To avoid possible confusions, we will always adopt the different notations with parentheses and with sums to distinguish between an Arakelov $\mathbb{R}$-divisor $(D,(g_\sigma)_\sigma)$ and its explicit form~$D+\sum c_\sigma X_\sigma$.
\end{nota}

An Arakelov $\mathbb R$-divisor is said to be \emph{vertical} if it is vertical as an arithmetic~$\mathbb{R}$-divisor.
The subspace of vertical Arakelov $\mathbb R$-divisors is denoted by~$\overline{\ver}_{\mathbb R}(X)$.
Hence one clearly has:
\begin{equation}\label{eq-vert-decomp}
\overline{\ver}_{\mathbb R}(X)
=
\bigoplus_{\mathfrak{p}\neq(0)}Z^1_\mathbb{R}(X)_{\mathfrak p}\oplus\left(\bigoplus_{\sigma\in\Sigma}\mathbb R X_\sigma\right)^{F_\infty}\,,
\end{equation}
where we recall that $Z_{\mathbb{R}}^1(X)_\mathfrak{p}$ stands for the real vector space of $\mathbb{R}$-divisors of $X$ with support contained in~$X_\mathfrak{p}$.

\begin{rem}
Recall that if $D$ is a vertical $\mathbb{R}$-divisor on~$X$, there is a canonical way to view it as an arithmetic $\mathbb{R}$-divisor by adjoining to it the vanishing current on $X_\sigma^{\an}$ for all~$\sigma$.
The obtained arithmetic $\mathbb{R}$-divisor is of Arakelov type and its explicit form is~$D$.
Hence, the abuse of notation introduced in Remark~\ref{canon_image_vert_arithm} is compatible with the identification given by Proposition~\ref{classical_ar}.
\end{rem}

We conclude this subsection with an application of Dirichlet's unit theorem, which shows a useful relation in~$\overline{\CH}^1_{\mathbb{R}}(X)$.

\begin{prop}\label{push_to_infty}
For any prime $\mathfrak p\neq(0)$ in~$B$ and any embedding~$\tau\in\Sigma$ there is a real number $c\neq0$ such that the Arakelov $\mathbb{R}$-divisor $X_{\mathfrak p}$ is equal to~$cX_\tau+cX_{\overline{\tau}}$ in~$\overline{\CH}^1_{\mathbb R}(X)$.
\end{prop}
\begin{proof}
Denoting by $h$ the class number of~$K$, we have by definition that $\mathfrak{p}^h=(\xi)$ for a nonzero element~$\xi\in O_K\setminus O_K^\times$.
Hence, by Remark~\ref{rem_div_const_expl},
\begin{equation}\label{eq_div_const}
    \widehat{\divi}(\xi)= hX_{\mathfrak{p}} - \sum_{\sigma\in\Sigma}\log|\sigma(\xi)|^2X_\sigma\,.
\end{equation}
Set now
\[
c:=\frac{1}{2h}\,\sum_{\sigma\in\Sigma}\log|\sigma(\xi)|^2\,.
\]
Notice that~$c\neq0$, as the product formula would otherwise imply that~$\xi\in O_K^\times$.
Moreover, equality \eqref{eq_div_const} ensures that the difference $X_{\mathfrak{p}}-cX_\tau-cX_{\overline{\tau}}$ can be represented in~$\overline{\CH}_{\mathbb{R}}^1(X)$ by the Arakelov $\mathbb{R}$-divisor
\[
\frac{1}{h}\sum_{\sigma\in\Sigma}\log|\sigma(\xi)|^2X_\sigma-cX_\tau-cX_{\overline{\tau}}\,,
\]
which has coefficients summing to zero.
The claim follows hence from the fact that such Arakelov $\mathbb{R}$-divisor is a real combination of principal ones, which can be shown by using Dirichlet's unit theorem as in the last part of the proof of Theorem~\ref{hi_bo} or in Example~\ref{exa_spec_O_K}.
\end{proof}

\subsection{The arithmetic Shioda-Tate formula}

We prove in this subsection our arithmetic Shioda-Tate formula relating the dimension of the first Arakelov-Chow vector space with the Mordell-Weil rank of the Albanese variety of~$X_K$.
In doing this, we exhibit an explicit basis of~$\overline{\CH}^1_{\mathbb{R}}(X)$. Notice that, throughout all the proof, we will use the identification between an Arakelov $\mathbb{R}$-divisor of $X$ and its explicit form provided by Proposition~\ref{classical_ar}.

\vspace{\baselineskip}

Our first step is to introduce a distinguished subspace $\mathcal{S}$ of the first Arakelov-Chow space and to compute its dimension in terms of the geometry of~$X$.
To do this, consider the generic fiber $X_K$ of~$X$, which is by assumption a smooth and geometrically integral projective variety over~$K$.
Its \emph{N\'eron-Severi group} $\NS(X_K)$ is by definition the group of classes of divisors of $X_K$ modulo algebraic equivalence in the sense of \cite[A.9.39]{BG06}.
It follows then from the theorem of the base, see for instance \cite[XIII, Th{\'e}or{\`e}me~5.1]{SGA}, that $\NS(X_K)$ is an abelian group of finite rank, which we denote by~$r$ from now on.

By taking a set of independent generators for the free part of~$\NS(X_K)$, considering arbitrary preimages in $\Div(X_K)$ and finally taking the Zariski closure in~$X$, one obtains a family
\begin{equation}\label{def_family_W}
\{W_1,\ldots,W_r\}    
\end{equation}
of divisors on~$X$, that we suppose fixed for the remaining of the section.
Denote by $W_{1,K},\ldots,W_{r,K}$ their respective restrictions to the generic fiber of~$X$; by construction, they are divisors of $X_K$ whose classes in $\NS(X_K)\otimes\mathbb{R}$ are a basis of this real vector space.

Also, in view of Proposition~\ref{classical_ar}, the divisors $W_1,\ldots,W_r$ can be seen as Arakelov $\mathbb{R}$-divisors on~$X$.
They are not vertical, as their restriction to the generic fiber cannot vanish by construction.
Finally, denote by $\mathcal{S}$ the image of
\begin{equation}\label{def_of_S}
    \overline{\ver}_{\mathbb R}(X)+\bigoplus_{i=1}^r\mathbb{R}W_i  \subseteq\overline{Z}^1_{\mathbb{R}}(X)
\end{equation}
in the vector space~$\overline{\CH}^1_{\mathbb R}(X)$.

\begin{lem}\label{lemma-dim_of_S}
In the above notations, $\mathcal{S}$ is a finite dimensional vector subspace of~$\overline{\CH}^1_{\mathbb R}(X)$.
Moreover,
\[
\dim(\mathcal S)=\rk(\NS(X_K))+1+\sum_{\mathfrak p\neq(0)}(f_{\mathfrak p}-1)\,,
\]
where $f_{\mathfrak{p}}$ denotes the number of irreducible components of the fiber~$X_{\mathfrak{p}}$. 
\end{lem}
\begin{proof}
Observe first that the images of $\overline{\ver}_{\mathbb R}(X)$ and $\bigoplus\mathbb{R}W_i$ in the first Arakelov-Chow space of $X$ are in direct sum.
Indeed, consider an element in the intersection of the image of $\bigoplus\mathbb{R}W_i$ and the image of a vertical Arakelov~$\mathbb{R}$-divisor.
It is the class of an Arakelov $\mathbb{R}$-divisor~$\sum_i\alpha_iW_i$ for which there is $\sum_{j=1}^k \lambda_j\widehat{\divi}(f_j)\in\widehat{\Rat}^1_{\mathbb R}(X)$ such that
\[
\sum_{i=1}^r\alpha_iW_i+\sum_{j=1}^k \lambda_j\widehat{\divi}(f_j)
\]
is vertical.
By definition, this implies that on the generic fibre $X_K$
\[
\sum_{i=1}^r\alpha_iW_{i,K}+\sum_{j=1}^k\lambda_j\divi(f_j)=0\,,
\]
and in turn that $\sum\alpha_iW_{i,K}=0$ in~$\NS(X_K)\otimes\mathbb{R}$.
The choice of $W_1,\ldots,W_r$ allows to conclude that~$\alpha_1=\ldots=\alpha_r=0$.

The same argument shows in fact that the images of $W_1,\ldots,W_r$ in $\overline{\CH}^1_{\mathbb R}(X)$ are linearly independent; hence, $\mathcal{S}$ is finite dimensional if the image of $\overline{\ver}_{\mathbb R}(X)$ is so, in which case
\begin{equation}\label{eq_dimS}
\dim(\mathcal{S})=\rk(\NS(X_K))+\dim(\Ima(\overline{\ver}_{\mathbb R}(X)))\,.
\end{equation}

So, in order to prove the lemma, it is enough to show that the second summand in the right hand side of \eqref{eq_dimS} is finite and that it equals~$1+\sum_{\mathfrak p}(f_{\mathfrak p}-1)$.
To do this, we exhibit an explicit basis of the involved vector space.

For all~$\mathfrak{p}\neq(0)$, denote by $\Gamma_1^{\mathfrak{p}},\ldots,\Gamma_{f_\mathfrak{p}}^{\mathfrak{p}}$ the irreducible components of the fiber of $X$ over~$\mathfrak{p}$.
Recalling~\eqref{eq-vert-decomp}, we have the decomposition
\[
\overline{\ver}_{\mathbb R}(X)
=
\bigoplus_{\mathfrak{p}\neq(0)}\left(\bigoplus_{i=1}^{f_{\mathfrak{p}}}\mathbb{R}\Gamma_i^{\mathfrak{p}}\right)\oplus\left(\bigoplus_{\sigma\in\Sigma}\mathbb R X_\sigma\right)^{F_\infty}\,.
\]
Fix an embedding~$\tau\in\Sigma$.
We claim that the image $\mathscr{B}$ in $\overline{\CH}^1_{\mathbb R}(X)$ of the family
\begin{equation}\label{basis_vert}
    \left\{X_\tau+ X_{\overline{\tau}}\right\}\,\cup\,\bigcup_{\mathfrak{p}\neq(0)}\left\{\Gamma^{\mathfrak p}_2,\Gamma^{\mathfrak p}_3,\ldots\Gamma^{\mathfrak p}_{f\mathfrak p}\right\}
\end{equation}
is a basis for~$\Ima(\overline{\ver}_{\mathbb R}(X))$.

Let us first show that $\mathscr{B}$ is a set of generators.
As $\Ima(\overline{\ver}_{\mathbb R}(X))$ is clearly generated by the images of the irreducible components of the finite fibers and by the images of the fibers at infinity, it is enough to show that all these elements can be expressed in~$\overline{\CH}^1_{\mathbb R}(X)$ as linear combinations of the ones appearing in~\eqref{basis_vert}.
First, for all~$\mathfrak{p}\neq(0)$, Proposition~\ref{push_to_infty} implies that in the first Arakelov-Chow space one has
\[
\sum_{i=1}^{f_\mathfrak{p}}m_i\Gamma_i^{\mathfrak{p}_i}=X_\mathfrak{p}=c(X_\tau+X_{\overline{\tau}})
\]
for some~$m_1,\ldots,m_{f_\mathfrak{p}}\in\mathbb{N}_{\geq1}$, which implies that $\Gamma_1^{\mathfrak{p}}$ belongs to the linear span of~$\mathscr{B}$.
Regarding the infinite fibers, for all $\sigma\in\Sigma$ one has that $X_\sigma+X_{\overline{\sigma}}$ is equal to $X_\tau+X_{\overline{\tau}}$ in $\overline{\CH}^1_{\mathbb R}(X)$ because of Dirichlet's unit theorem.

Secondly, let us prove that the elements of $\mathscr{B}$ are linearly independent.
Suppose by contradiction that there is a nontrivial linear relation
\[
a_0(X_\tau+X_{\overline{\tau}})+\sum_{\mathfrak{p}\neq(0)}\sum_{i=2}^{f_{\mathfrak{p}}}a_{\mathfrak{p},i}\Gamma_i^{\mathfrak{p}}=0
\]
in~$\overline{\CH}^{1}_{\mathbb R}(X)$.
Since $X_\tau+X_{\overline{\tau}}$ is not zero in the first Arakelov-Chow space (for instance, it is not numerically trivial), there exists $\mathfrak{q}\neq(0)$ such that at least one of the real numbers $a_{\mathfrak{q},2},\ldots,a_{\mathfrak{q},f_{\mathfrak{q}}}$ is nonzero.
Set
\[
W:=-\sum_{i=2}^{f_{\mathfrak{q}}}a_{\mathfrak{q},i}\Gamma_i^{\mathfrak{q}},
\qquad\text{and}\qquad
\overline{Z}:=a_0(X_\tau+X_{\overline{\tau}})+\sum_{\mathfrak{p}\neq\mathfrak{q}}\sum_{i=2}^{f_{\mathfrak{p}}}a_{\mathfrak{p},i}\Gamma_i^{\mathfrak{p}}\,.
\]
Choose an ample arithmetic divisor $\widehat{A}$ on~$X$.
Since $\overline{Z}$ and $W$ are numerically equivalent, we deduce from the associativity of the arithmetic intersection product and from Remark~\ref{intersect_vert_fib} that
\[
\left<W,W\right>_{\mathfrak  q,\widehat {A}}
=
\left< W,W\cdot\widehat{A}^{d-1}\right>
=
\left< \overline{Z},W\cdot\widehat{A}^{d-1}\right>
=
\left< \overline{Z}\cdot W,\widehat{A}^{d-1}\right>
=
0\,.
\]
Then, by Proposition~\ref{hi_li} it must happen that $W$ is a real multiple of~$X_\mathfrak{q}$.
Since $\Gamma_1^{\mathfrak{q}}$ does not appear in the support of~$W$, while at least one of the other irreducible components does by the assumption on~$\mathfrak{q}$, we get a contradiction.

To conclude, observe that the geometric irreducibility of $X_K$ implies that $f_{\mathfrak{p}}=1$ for all but finitely many~$\mathfrak{p}\in B$, see for instance~\cite[Proposition~9.7.8]{EGA-IV}.
Hence, the family $\mathscr{B}$ is finite.
Since the images in the first Arakelov-Chow space of the elements in \eqref{basis_vert} are different (we have in fact proved that they are linearly independent), we infer that the cardinality of $\mathscr{B}$ is~$1+\sum_{\mathfrak p}(f_{\mathfrak p}-1)$, concluding the proof.
\end{proof}

Before stating our final result, we recall that $\Pic^0(X_K)$ is defined as the quotient of divisors algebraically equivalent to zero on $X_K$ by principal divisors.

We now recall some geometric constructions allowing to associate two abelian varieties defined over $K$ to the generic fiber $X_K$ of~$X$.
First, it would be desirable to have an abelian variety over $K$ whose $F$-points agree with~$\Pic^0(X_F)$ for all field extensions $F$ of~$K$.
This holds true when~$X_K(K)\neq\emptyset$, see \cite[Corollary~8.4.10]{BG06}; in the general case, the situation is more complicated.

A crucial observation is that, even if $X_K$ has not a $K$-rational point, the structure morphism~$X_K\to\spec K$ always admits a section``locally for the fppf topology on~$\spec K$", suggesting that an abelian variety could solve the moduli problem on such a site.
In fact, \cite[Th\'eor\`eme~3.1]{Gro62a} shows that under our assumptions on $X_K$ there exists a scheme $\bPic(X_K)$ which represents the sheaf associated to the relative Picard functor in the fppf topology, see also \cite{Kle05} for a modern presentation and for the precise definitions.
However, even if it happens that $\Pic(X_K)$ embeds naturally into~$\bPic(X_K)(K)$, such an embedding is not necessarily an isomorphism, see \cite[Theorem~9.2.5, Exercise~9.3.11 and Exercise~9.2.4]{Kle05}.

More interestingly for our treatment, by \cite[Corollaire~3.2 and Th\'eor\`eme~3.3.(i)]{Gro62b} there exists a unique abelian subscheme $\bPic^0(X_K)$ of $\bPic(X_K)$ whose underlying set agrees with the connected component of the identity.
It turns out that $\bPic^0(X_K)$ is a separated geometrically irreducible projective group subscheme of finite type of $\bPic(X_K)$ \cite[Proposition~5.3 and Theorem~5.4]{Kle05}.
Moreover, by \cite{Oor66} its base change to~$\overline{K}$, which is a group scheme over a field of characteristic zero, is reduced.
Therefore, $\bPic^0(X_K)$ is also geometrically reduced and hence it is an abelian variety, which is classically referred to as the \emph{Picard variety} of~$X_K$.

The $K$-points of the Picard variety of $X_K$ form a group which is not too far from the actual~$\Pic^0(X_K)$, as the next lemma shows.

\begin{lem}\label{2tor}
In the above assumptions and notations, $\Pic^0(X_K)$ is a subgroup of~$\bPic^0(X_K)(K)$, and the quotient
\[
\bPic^0(X_K)(K)/\Pic^0(X_K)
\]
is a torsion group.
\end{lem}
\begin{proof}
First of all, by seeing $\Pic(X_K)$ in~$\bPic(X_K)(K)$, \cite[Proposition~9.5.10]{Kle05} affirms that~$\Pic^0(X_K)=\bPic^0(X_K)(K)\cap\Pic(X_K)$.
In particular, the claimed subgroup relation holds, and there is an embedding
\[
\bPic^0(X_K)(K)/\Pic^0(X_K)\hookrightarrow\bPic(X_K)(K)/\Pic(X_K)\,.
\]
Then, by \cite[Corollaire~5.3]{Gro66} we have an exact sequence
\[
0\to\Pic(X_K)\to\bPic(X_K)(K)\to\Br(K)\,,
\]
which identifies $\bPic(X_K)(K)/\Pic(X_K)$ with a subgroup of the Brauer group $\Br(K)$ of~$K$. 
Now, since $K$ is a number field, $\Br(K)$ can be embedded in $\bigoplus_{v} \Br(K_{v})$, where $v$ ranges over all places, and the latter is a torsion group, see \cite[Corollaire~2.3 and pages 94-95]{Gro66} to check these claims. So 
$\bPic^0(X_K)(K)/\Pic^0(X_K)$ is a torsion group as well.
\end{proof}

With the previous in mind, the \emph{Albanese variety} $\alb(X_K)$ of $X_K$ is by definition the dual abelian variety of~$\bPic^0(X_K)$, and it satisfies a universal property as in~\cite[Th\'eor\`eme~3.3.(iii)]{Gro62b}.

We can finally state the main result of this section.

\begin{thm}[arithmetic Shioda-Tate formula]\label{hi_ST}
Let $X$ be a projective arithmetic variety over the ring of integers of a number field~$K$.
The first Arakelov-Chow vector space of $X$ is finite dimensional and:
\[
\dim(\overline{\CH}^1_{\mathbb R}(X))=\rk(\alb(X_K)(K))+\rk(\NS(X_K))+1+\sum_{\mathfrak p\neq (0)} (f_{\mathfrak p}-1)\,,
\]
where $f_\mathfrak{p}$ stands for the number of irreducible components of the fiber of $X$ over~$\mathfrak{p}$, whereas $\alb(X_K)$ and $\NS(X_K)$ denote respectively the Albanese variety and the N\'eron-Severi group of the generic fiber~$X_K$. 
\end{thm}
\begin{proof}
Consider the real vector subspaces of $Z_{\mathbb{R}}^1(X_K)$ generated by algebraically trivial divisors and by principal divisors on~$X_K$ respectively, and denote by $\Pic_{\mathbb{R}}^0(X_K)$ their quotient.
Notice that, as $\mathbb{R}$ is a flat~$\mathbb{Z}$-module, one has that~$\Pic_{\mathbb{R}}^0(X_K)\simeq\Pic^0(X_K)\otimes\mathbb{R}$.
Also, recall that in \eqref{def_family_W} we have fixed a family $\{W_1,\ldots,W_r\}$ of divisors on $X$ inducing a basis of~$\NS(X_K)\otimes\mathbb{R}$. 
The core of the proof is to define and study the linear map
\begin{eqnarray*}
\psi\colon\overline{\CH}^1_{\mathbb R}(X)&\longrightarrow& \Pic_{\mathbb R}^0(X_K)\\
\overline D &\longmapsto & D_K-\sum^r_{i=1} a_i(D)W_{i,K}\,,
\end{eqnarray*}
where $a_1(D),\ldots,a_r(D)$ are the unique real coefficients needed to express the class of $D_K$ in $\NS(X_K)\otimes\mathbb{R}$ in terms of the basis determined by~$\{W_{1,K},\ldots,W_{r,K}\}$.
By construction, the image $\psi(\overline{D})$ is an element of $\Pic_{\mathbb R}^0(X_K)$.
Moreover, if $\overline{D}$ is $0$ in $\overline{\CH}^1_{\mathbb R}(X)$, we have that $D_K$ is a linear real combination of principal divisors on~$X_K$, so $\psi$ is well-defined.
Observe that the map $\psi$ is surjective; indeed, for an element in~$\Pic_{\mathbb{R}}^0(X_K)$ it is enough to consider its Zariski closure in~$X$ to obtain a preimage.
We now show that~$\ker{\psi}=\mathcal S$, where $\mathcal{S}$ is the image of the subspace \eqref{def_of_S} in~$\overline{\CH}^1_{\mathbb{R}}(X)$.
The inclusion $\ker{\psi}\supseteq \mathcal S$ is clear from the definitions.
Consider now~$\overline D\in \ker{\psi}$; then there exist $f_1,\ldots,f_k\in K(X_K)^\times$  such that
\[
D_K-\sum^r_{i=1} a_i(D)W_{i,K}=\sum_{j=1}^k\lambda_j\divi(f_j)
\]
for some real numbers~$\lambda_1,\ldots,\lambda_k$.
Seeing $f_1,\ldots,f_k$ as nonzero rational functions on~$X$, we have that the Arakelov $\mathbb{R}$-divisor
\[
\sum_{j=1}^k\lambda_j\widehat{\divi}(f_j)-\overline D+\sum^r_{i=1} a_i(D)W_i
\]
has trivial generic fiber, which in turn implies that it is vertical.
Therefore~$\overline D\in S$.

We have shown that $\psi$ induces the following short exact sequence of vector spaces:
\[
0\longrightarrow\mathcal S\longrightarrow \overline{\CH}^1_{\mathbb R}(X)\longrightarrow  \Pic_{\mathbb{R}}^0(X_K)\longrightarrow 0\,.
\]
The real dimension of the space $\Pic_{\mathbb{R}}^0(X_K)$ coincides with the rank of the group $\Pic^0(X_K)$ and hence with the Mordell-Weil rank of the $K$-variety~$\bPic^0(X_K)$ thanks to Lemma~\ref{2tor}.
Since an abelian variety and its dual are always isogenous over~$K$ (see for instance \cite[Theorem~8.5.1 and the proof of Corollary~8.5.11]{BG06}) and the Mordell-Weil rank of an abelian variety is finite and invariant under isogenies over~$K$, we have that~$\rk(\bPic^0(X_K)(K))=\rk(\alb(X_K)(K))$ is finite.
Therefore the statement follows from Lemma~\ref{lemma-dim_of_S}.
\end{proof}

\begin{rem}
Let $s$ denote the rank of the Mordell-Weil group of the Picard variety of~$X_K$, and let $D_1,\ldots,D_s$ be a family of algebraically trivial divisors on $X_K$ whose classes generate the free part of~$\Pic^0(X_K)$.
By the proof of Theorem~\ref{hi_ST} and Lemma~\ref{lemma-dim_of_S}, an explicit basis of the first Arakelov-Chow space of~$X$ is given by the following Arakelov~$\mathbb{R}$-divisors, written in their explicit form: the Zariski closures of $D_1,\ldots,D_s$ in~$X$, the divisors $W_1,\ldots,W_r$ of~\eqref{def_family_W}, and the elements in~\eqref{basis_vert}.
\end{rem}

\begin{rem}
The dimension of the first Arakelov-Chow space calculated in Theorem~\ref{hi_ST} is independent of the initial choice of the K\"ahler forms on each~$X_\sigma^{\an}$.
This is certainly not surprising because of \cite[Theorem in~5.1.6, entry~(i)]{GS90a}.
\end{rem}

We conclude the paper by presenting some instances of the arithmetic Shioda-Tate formula proved in Theorem~\ref{hi_ST}.

\begin{exa}[low dimensions]
When~$X=\spec O_K$, the Picard group of $X_K$ is trivial, hence both groups $\NS(X_K)$ and $\alb(X_K)(K)$ are so.
Moreover, all special fibers of $X$ are irreducible as $\mathfrak{p}$ is prime, hence Theorem~\ref{hi_ST} affirms that
\[
\dim\left(\overline{\CH}^1_{\mathbb R}(\spec O_K)\right)=1\,,
\]
which agrees with the direct calculation of Example~\ref{exa_ar_spec_O_K}.

When $X$ has relative dimension $1$ over~$\spec O_K$, the generic fiber $X_K$ is a curve over $K$ and $\alb(X_K)$ is by definition its Jacobian.
Noticing that $\NS(X_K)$ has rank $1$ since $X_K$ is a curve, the statement of Theorem~\ref{hi_ST} is in this case a direct consequence of~\cite[Theorem~4.(d)]{Fal84}.
\end{exa}

\begin{exa}[projective spaces]
Let $X=\mathbb{P}_{O_K}^d$ be the projective space over $\spec O_K$ of relative dimension~$d$.
All its fibers are irreducible since they  are projective spaces over a finite field, and the generic fiber is~$\mathbb{P}^d_K$.
Remember that $\Pic(\mathbb{P}^d_K)\simeq \mathbb Z$ and each element of  $\Pic(\mathbb{P}^d_K)$ can be identified with the isomorphism class of the sheaf $\mathscr{O}_{\mathbb {P}^d_K}(m)$ for~$m\in\mathbb Z$.
It follows that the numerically trivial divisors of $\Pic(\mathbb{P}^d_K)$ are exactly the principal divisors; therefore $\Pic^0(\mathbb{P}^d_K)=\{0\}$ and we can conclude that $\NS(\mathbb{P}^d_K)\simeq\mathbb{Z}$ whereas $\alb(\mathbb{P}^d_K)$ is trivial.
By Theorem \ref{hi_ST} we deduce that $\overline{\CH}^1_{\mathbb R}(\mathbb P^d_{O_K})$ has dimension~$2$, as expected.
\end{exa}

\begin{exa}[abelian varieties]
Let $A$ be a regular integral projective model of an abelian variety defined over the number field~$K$.
Then, Theorem \ref{hi_ST} reads
\[
\dim(\overline{\CH}^1_{\mathbb R}(A))
=
\rk(A_K(K))+\rk(\NS(A_K))+1+\sum_{\mathfrak p\neq (0)} (f_{\mathfrak p}-1)\,,
\]
giving a relation between the dimension of the first Arakelov-Chow space of $A$ and the Mordell-Weil rank of its generic fiber.
\end{exa}

\bibliographystyle{my_style}
\bibliography{biblio}

\begin{thebibliography}{SABK94}

\bibitem[Ara74]{Ara74}
S.~J. Arakelov.
\newblock Intersection theory of divisors on an arithmetic surface.
\newblock {\em Mathematics of the USSR-Izvestiya}, 8(6):1167, 1974.

\bibitem[BC09]{BC09}
J.-B. Bost and A.~{Chambert-Loir}.
\newblock Analytic curves in algebraic varieties over number fields.
\newblock In {\em Algebra, arithmetic, and geometry: in honor of {Y}u. {I}.
  {M}anin. {V}ol. {I}}, volume 269 of {\em Progr. Math.}, pages 69--124.
  Birkh\"{a}user Boston, Boston, MA, 2009.

\bibitem[BG06]{BG06}
E.~Bombieri and W.~Gubler.
\newblock {\em Heights in {D}iophantine geometry}, volume~4 of {\em New
  Mathematical Monographs}.
\newblock Cambridge University Press, Cambridge, 2006.

\bibitem[BGS94]{BGS94}
J.-B. Bost, H.~Gillet, and C.~Soul\'e.
\newblock Heights of projective varieties and positive {G}reen forms.
\newblock {\em J. Amer. Math. Soc.}, 7(4):903--1027, 1994.

\bibitem[CD20]{CD20}
W.~Czerniawska and P.~Dolce.
\newblock Adelic geometry on arithmetic surfaces {II}: {C}ompleted adeles and
  idelic {A}rakelov intersection theory.
\newblock {\em J. Number Theory}, 211:235--296, 2020.

\bibitem[Cha17]{Cha17}
F.~Charles.
\newblock Arithmetic ampleness and an arithmetic {B}ertini theorem.
\newblock \url{https://arxiv.org/abs/1703.02481}, 2017.

\bibitem[deR84]{deR84}
G.~de~Rham.
\newblock {\em Differentiable manifolds}, volume 266 of {\em Grundlehren der
  Mathematischen Wissenschaften [Fundamental Principles of Mathematical
  Sciences]}.
\newblock Springer-Verlag, Berlin, 1984.
\newblock Forms, currents, harmonic forms, Translated from the French by F. R.
  Smith, With an introduction by S. S. Chern.

\bibitem[Fal83]{Fal83}
G.~Faltings.
\newblock Endlichkeitss\"{a}tze f\"{u}r abelsche {V}ariet\"{a}ten \"{u}ber
  {Z}ahlk\"{o}rpern.
\newblock {\em Invent. Math.}, 73(3):349--366, 1983.

\bibitem[Fal84]{Fal84}
G.~Faltings.
\newblock Calculus on arithmetic surfaces.
\newblock {\em Ann. of Math. (2)}, 119(2):387--424, 1984.

\bibitem[Ful98]{Ful98}
W.~Fulton.
\newblock {\em Intersection theory}, volume~2 of {\em Ergebnisse der Mathematik
  und ihrer Grenzgebiete. 3. Folge. A Series of Modern Surveys in Mathematics
  [Results in Mathematics and Related Areas. 3rd Series. A Series of Modern
  Surveys in Mathematics]}.
\newblock Springer-Verlag, Berlin, second edition, 1998.

\bibitem[Gro62a]{Gro62a}
A.~Grothendieck.
\newblock Technique de descente et th\'{e}or\`emes d'existence en
  g\'{e}om\'{e}trie alg\'{e}brique. {V}. {L}es sch\'{e}mas de {P}icard:
  th\'eor\`emes d'existence.
\newblock In {\em S\'{e}minaire {B}ourbaki, {V}ol. 7}, pages Exp. No. 232,
  143--161. Soc. Math. France, Paris, 1962.

\bibitem[Gro62b]{Gro62b}
A.~Grothendieck.
\newblock Technique de descente et th\'{e}or\`emes d'existence en
  g\'{e}om\'{e}trie alg\'{e}brique. {VI}. {L}es sch\'{e}mas de {P}icard:
  propri\'{e}t\'{e}s g\'{e}n\'{e}rales.
\newblock In {\em S\'{e}minaire {B}ourbaki, {V}ol. 7}, pages Exp. No. 236,
  221--243. Soc. Math. France, Paris, 1962.

\bibitem[EGA-IV]{EGA-IV}
A.~Grothendieck.
\newblock \'{E}l\'ements de g\'eom\'etrie alg\'ebrique : {IV}. \'{E}tude locale
  des sch\'emas et des morphismes de sch\'emas, {T}roisi\`eme partie.
\newblock {\em Publications Math\'ematiques de l'IH\'ES}, 28:5--255, 1966.

\bibitem[Gro66]{Gro66}
A.~Grothendieck.
\newblock Groupe de {B}rauer. {III}. {E}xemples et compl{\'e}ments.
\newblock In {\em Dix Exposes Sur La Cohomologie Des Schemas}, pages 88--188.
  North-Holland Publishing Company, 1966.

\bibitem[GS90a]{GS90a}
H.~Gillet and C.~Soul\'{e}.
\newblock Arithmetic intersection theory.
\newblock {\em Inst. Hautes \'{E}tudes Sci. Publ. Math.}, (72):93--174 (1991),
  1990.

\bibitem[GS90b]{GS90b}
H.~Gillet and C.~Soul\'{e}.
\newblock Characteristic classes for algebraic vector bundles with {H}ermitian
  metric. {I}.
\newblock {\em Ann. of Math. (2)}, 131(1):163--203, 1990.

\bibitem[GS94]{GS94}
H.~Gillet and C.~Soul\'{e}.
\newblock Arithmetic analogs of the standard conjectures.
\newblock In {\em Motives ({S}eattle, {WA}, 1991)}, volume~55 of {\em Proc.
  Sympos. Pure Math.}, pages 129--140. Amer. Math. Soc., Providence, RI, 1994.

\bibitem[HPW05]{HPW05}
M.~Hindry, A.~Pacheco, and R.~Wazir.
\newblock Fibrations et conjecture de {T}ate.
\newblock {\em J. Number Theory}, 112(2):345--368, 2005.

\bibitem[Hri85]{Hri85}
P.~Hriljac.
\newblock Heights and {A}rakelov's intersection theory.
\newblock {\em Amer. J. Math.}, 107(1):23--38, 1985.

\bibitem[Kue95]{Kue95}
K.~K\"{u}nnemann.
\newblock Some remarks on the arithmetic {H}odge index conjecture.
\newblock {\em Compositio Math.}, 99(2):109--128, 1995.

\bibitem[Kah09]{Kah09}
B.~Kahn.
\newblock D\'{e}monstration g\'{e}om\'{e}trique du th\'{e}or\`eme de
  {L}ang-{N}\'{e}ron et formules de {S}hioda-{T}ate.
\newblock In {\em Motives and algebraic cycles}, volume~56 of {\em Fields Inst.
  Commun.}, pages 149--155. Amer. Math. Soc., Providence, RI, 2009.

\bibitem[Kle05]{Kle05}
S.~L. Kleiman.
\newblock The {P}icard scheme.
\newblock In {\em Fundamental algebraic geometry}, volume 123 of {\em Math.
  Surveys Monogr.}, pages 235--321. Amer. Math. Soc., Providence, RI, 2005.

\bibitem[KM00]{KM00}
K.~K\"{u}nnemann and V.~Maillot.
\newblock Th\'{e}or\`emes de {L}efschetz et de {H}odge arithm\'{e}tiques pour
  les vari\'{e}t\'{e}s admettant une d\'{e}composition cellulaire.
\newblock In {\em Regulators in analysis, geometry and number theory}, volume
  171 of {\em Progr. Math.}, pages 197--205. Birkh\"{a}user Boston, Boston, MA,
  2000.

\bibitem[Liu02]{Liu02}
Q.~Liu.
\newblock {\em Algebraic geometry and arithmetic curves}, volume~6 of {\em
  Oxford Graduate Texts in Mathematics}.
\newblock Oxford University Press, Oxford, 2002.
\newblock Translated from the French by Reinie Ern\'{e}, Oxford Science
  Publications.

\bibitem[LN59]{LN59}
S.~Lang and A.~N\'{e}ron.
\newblock Rational points of abelian varieties over function fields.
\newblock {\em Amer. J. Math.}, 81:95--118, 1959.

\bibitem[Mai00]{Mai00}
V.~Maillot.
\newblock G\'eom\'etrie d'{A}rakelov des vari\'et\'es toriques et fibr\'es en
  droites int\'egrables.
\newblock {\em M\'em. Soc. Math. Fr. (N.S.)}, (80):vi+129, 2000.

\bibitem[Mor96]{Mor96}
A.~Moriwaki.
\newblock Hodge index theorem for arithmetic cycles of codimension one.
\newblock {\em Math. Res. Lett.}, 3(2):173--183, 1996.

\bibitem[Mor12]{Mor12}
A.~Moriwaki.
\newblock {Zariski decompositions on arithmetic surfaces}.
\newblock {\em {Publ. Res. Inst. Math. Sci.}}, 48(4):799--898, 2012.

\bibitem[Mor13]{Mor13}
A.~Moriwaki.
\newblock Toward {D}irichlet's unit theorem on arithmetic varieties.
\newblock {\em Kyoto J. Math.}, 53(1):197--259, 2013.

\bibitem[Mor14]{Mor14}
A.~Moriwaki.
\newblock {\em Arakelov geometry}.
\newblock American Mathematical Society, Providence, Rhode Island, 2014.

\bibitem[Ogu09]{Ogu09}
K.~Oguiso.
\newblock Shioda-{T}ate formula for an abelian fibered variety and
  applications.
\newblock {\em J. Korean Math. Soc.}, 46(2):237--248, 2009.

\bibitem[Oor66]{Oor66}
F.~Oort.
\newblock Algebraic group schemes in characteristic zero are reduced.
\newblock {\em Invent. Math.}, 2:79--80, 1966.

\bibitem[Rud91]{Rud91}
W.~Rudin.
\newblock {\em Functional analysis}.
\newblock International Series in Pure and Applied Mathematics. McGraw-Hill,
  Inc., New York, second edition, 1991.

\bibitem[SABK94]{SABK94}
C.~Soul{\'e}, D.~Abramovich, J.F. Burnol, and J.~Kramer.
\newblock {\em Lectures on Arakelov geometry}.
\newblock Cambridge University Press, Cambridge New York, 1994.

\bibitem[SGA]{SGA}
{\em Th\'{e}orie des intersections et th\'{e}or\`eme de {R}iemann-{R}och}.
\newblock Lecture Notes in Mathematics, Vol. 225. Springer-Verlag, Berlin-New
  York, 1971.
\newblock S\'{e}minaire de G\'{e}om\'{e}trie Alg\'{e}brique du Bois-Marie
  1966--1967 (SGA 6), Dirig\'{e} par P. Berthelot, A. Grothendieck et L.
  Illusie. Avec la collaboration de D. Ferrand, J. P. Jouanolou, O. Jussila, S.
  Kleiman, M. Raynaud et J. P. Serre.

\bibitem[Shi72]{Shi72}
T.~Shioda.
\newblock On elliptic modular surfaces.
\newblock {\em J. Math. Soc. Japan}, 24:20--59, 1972.

\bibitem[Sta-Pr]{Sta-Pr}
The {Stacks Project Authors}.
\newblock \textit{Stacks Project}.
\newblock \url{https://stacks.math.columbia.edu}, 2018.

\bibitem[Tak98]{Tak98}
Y.~Takeda.
\newblock A relation between standard conjectures and their arithmetic
  analogues.
\newblock {\em Kodai Math. J.}, 21(3):249--258, 1998.

\bibitem[Tat95]{Tat95}
J.~Tate.
\newblock On the conjectures of {B}irch and {S}winnerton-{D}yer and a geometric
  analog.
\newblock In {\em S\'{e}minaire {B}ourbaki, {V}ol. 9}, pages Exp. No. 306,
  415--440. Soc. Math. France, Paris, 1995.

\bibitem[YZ17]{YZ17}
X.~Yuan and S.-W. Zhang.
\newblock The arithmetic {H}odge index theorem for adelic line bundles.
\newblock {\em Math. Ann.}, 367(3-4):1123--1171, 2017.

\bibitem[Zha95]{Zha95}
S.-W. Zhang.
\newblock Positive line bundles on arithmetic varieties.
\newblock {\em J. Amer. Math. Soc.}, 8(1):187--221, 1995.

\end{thebibliography}

\Addresses

\end{document}